\documentclass[11pt]{article}

\usepackage[T1]{fontenc}
\usepackage[utf8]{inputenc}
\usepackage{lmodern}
\usepackage{amsmath,amssymb,amsthm,mathtools}
\usepackage[margin=1in]{geometry}
\usepackage{microtype}
\usepackage{booktabs,tabularx,array}
\usepackage{enumitem}
\usepackage{algorithm,algpseudocode}
\usepackage{needspace,etoolbox}
\usepackage[hidelinks]{hyperref}
\usepackage{bookmark}

\newcommand{\ResearchAgentSystem}{our laboratory's internal auto-research system}

\newif\ifanonymous
\anonymousfalse

\newcommand{\AuthorOne}{Haihan Zhang}
\newcommand{\AuthorTwo}{Wendao Wu}
\newcommand{\AuthorThree}{Chenheng Zhang}
\newcommand{\AuthorFour}{Haoxuan Li}
\newcommand{\AuthorFive}{Zhouchen Lin}
\newcommand{\AuthorSix}{Cong Fang}
\newcommand{\AffiliationOne}{Peking University}

\newcommand{\EmailOne}{colombo.2026@outlook.com}

\ifanonymous
  \newcommand{\PDFAuthors}{Anonymous Authors}
\else
  \newcommand{\PDFAuthors}{\AuthorOne; \AuthorTwo; \AuthorThree; \AuthorFour; \AuthorFive}
\fi

\hypersetup{
  pdftitle={Matching Higher-Order Oracle Complexity for Smooth Monotone Variational Inequalities},
  pdfauthor={\PDFAuthors},
  pdfsubject={Higher-order oracle complexity, tangent residuals, Halpern iteration, resisting-oracle lower bounds, and Lean-based review}
}

\allowdisplaybreaks[1]

\setlist{leftmargin=*,itemsep=2pt,topsep=3pt}
\numberwithin{equation}{section}

\newcommand{\R}{\mathbb R}
\newcommand{\E}{\mathbb E}
\newcommand{\Prob}{\mathbb P}
\newcommand{\eps}{\varepsilon}
\newcommand{\cO}{\mathcal O}
\newcommand{\Ot}{\widetilde{\mathcal O}}
\newcommand{\ip}[2]{\langle #1,#2\rangle}
\newcommand{\norm}[1]{\lVert #1\rVert}
\newcommand{\rtan}{r_{\mathrm{tan}}}
\newcommand{\Span}{\operatorname{span}}
\newcommand{\Lip}{\operatorname{Lip}}
\newcommand{\dist}{\operatorname{dist}}
\newcommand{\diam}{\operatorname{diam}}
\newcommand{\Ball}{B_2^d(R)}
\newcommand{\Proj}{\Pi_{\mathcal X}}
\newcommand{\LA}{\mathcal L_{\mathrm A}}
\newcommand{\LS}{\mathcal L_{\mathrm S}}
\newcommand{\Fail}{\textsc{Failure}}
\newcommand{\sym}{\operatorname{sym}}
\newcommand{\Nzero}{\mathbb N_0}
\newcommand{\Jet}{\mathcal J_{p,d}}
\newcommand{\Class}{\mathfrak F_{p,L_p}(\mathcal X)}
\newcommand{\Qtan}{Q_{\mathrm{tan}}}
\newcommand{\Qgap}{Q_{\mathrm{gap}}}
\newcommand{\TccTensor}[1]{\mathsf T^{\mathrm{tensor}}_{#1,\mathrm{cc}}}
\newcolumntype{P}[1]{>{\raggedright\arraybackslash}p{#1}}
\newcolumntype{Y}{>{\raggedright\arraybackslash}X}

\newtheorem{theorem}{Theorem}[section]
\newtheorem{lemma}[theorem]{Lemma}
\newtheorem{proposition}[theorem]{Proposition}
\newtheorem{corollary}[theorem]{Corollary}
\theoremstyle{definition}
\newtheorem{definition}[theorem]{Definition}

\theoremstyle{remark}
\newtheorem{remark}[theorem]{Remark}

\AtBeginEnvironment{theorem}{\Needspace{5\baselineskip}}
\AtBeginEnvironment{lemma}{\Needspace{5\baselineskip}}
\AtBeginEnvironment{proposition}{\Needspace{5\baselineskip}}
\AtBeginEnvironment{corollary}{\Needspace{5\baselineskip}}
\AtBeginEnvironment{definition}{\Needspace{5\baselineskip}}
\AtBeginEnvironment{assumption}{\Needspace{5\baselineskip}}
\AtBeginEnvironment{proof}{\Needspace{3\baselineskip}}

\title{\textbf{Matching Higher-Order Oracle Complexity for\\
Smooth Monotone Variational Inequalities}}

\ifanonymous
  \author{Anonymous Authors}
\else
  \author{%
    \AuthorOne$^{1,*}$ \quad
    \AuthorTwo$^{1,*}$ \quad
    \AuthorThree$^{1,*}$\\[0.20em]
    \AuthorSix$^{1,\dagger}$ \quad
    \AuthorFour$^{1,\dagger}$ \quad
    \AuthorFive$^{1,\dagger}$\\[0.60em]
    \small $^{1}$\AffiliationOne\\[0.35em]
    \small
    \href{mailto:\EmailOne}{\texttt{\EmailOne}} \quad
    \\[0.35em]
    \small $^{*}$Equal contribution.    \small $^{\dagger}$Corresponding authors.
  }
\fi
\date{}

\begin{document}
\maketitle

\begin{abstract}
We establish near-optimal higher-order oracle bounds for smooth
monotone variational inequalities. For fixed $p\ge2$, we consider a
monotone operator $F$ on a known compact convex set $\mathcal X$ of
diameter at most $D$, satisfying
$\operatorname{Lip}(D^{p-1}F)\le L_p$.
Each feasible query returns the complete jet
$(F,DF,\ldots,D^{p-1}F)$, and the goal is to find a point with tangent
residual
$\operatorname{dist}(0,F(x)+N_{\mathcal X}(x))\le\varepsilon$,
where $N_{\mathcal X}(x)$ denotes the normal cone. 
Writing $Q=L_pD^p/\varepsilon$, we improve the
$\widetilde O_p(Q^{1/p})$ upper bound of~\cite{halpern} to
$\widetilde O_p(Q^{2/(3p-1)})$ with a dimension-independent
deterministic algorithm.
Our algorithm returns an explicit tangent-residual certificate,
which also guarantees the benchmark's proximal-residual accuracy
at the same output point.
Thus, the faster rate applies to the original benchmark criterion,
without weakening the accuracy requirement.
We complement this improvement with an
$\Omega_p(Q^{2/(3p-1)})$ lower bound for arbitrary adaptive
deterministic algorithms and randomized algorithms with per-instance
success probability at least $2/3$, without span or tensor-update
restrictions. Together, these bounds establish
\[
  \widetilde\Theta_p\!\left(
    \left(\frac{L_pD^p}{\varepsilon}\right)^{2/(3p-1)}
  \right)
\]
as the high-dimensional worst-case oracle complexity.
The upper bound combines extrapolated warm starts for inexact
Halpern iteration with certified Taylor-model solves and a finite
acceleration hierarchy, counting all nested oracle evaluations.
The lower bound hides complete derivative tensors using exactly
flat smooth gates, a resisting-oracle construction, and random-frame
transcript coupling. 
Our method further applies to the smooth convex--concave minimax
setting of~\cite{minimax}. At fixed geometry and smoothness, it
improves the
$\widetilde O_p(\varepsilon^{-4/(3p+1)})$ upper bound
of~\cite{minimax} to
$\widetilde O_p(\varepsilon^{-2/(3p-1)})$,
matching the lower-bound exponent (\hyperref[cor:cc-specialization]
 {Corollary~\ref*{cor:cc-specialization}}).
\end{abstract}
\noindent\textbf{Keywords:} monotone variational inequalities; higher-order oracle complexity; matching upper and lower bounds; auto-research systems.

\paragraph{Publication Intent and Recognition of Independent Work.} 
This preprint is released as a research output of our study of
automated scientific research, for public scrutiny and discussion.
As we do not intend to submit this manuscript to a journal or
conference, it has not undergone the full editorial refinement
and internal review we would normally require for submission.
The authors remain responsible for its contents and for correcting
any errors or inaccurate attribution.

We recognize that releasing these results may affect researchers
who have independently been pursuing the same questions, including
work that is not yet public.
If this applies to your work, we invite you to contact us at
\texttt{colombo.2026@outlook.com}.
A brief description of the overlap, together with any dated
materials you are comfortable sharing, would help us understand
its scope and chronology.
Such materials might include draft excerpts, presentation slides,
or revision records; a complete manuscript or unpublished proof
is not required to begin the conversation.

We will review these cases carefully and update citations,
acknowledgments, and accounts of independent or concurrent work
as appropriate, including revising priority claims where warranted.
Nonpublic materials will be treated confidentially and used only
to assess the overlap unless their authors authorize otherwise.
We will obtain explicit permission from the authors before
publicly discussing such materials or providing them to an
automated research system.

\noindent\textbf{AI Usage.}
Nearly the entire research pipeline for this paper was carried out by  Colombo, 
\ResearchAgentSystem{}, powered by GPT-5.6 Sol. The system also conducted a
Lean-backed article audit of the resulting manuscript. The authors subsequently reviewed and approved the
mathematical claims, presentation, and formal artifacts, and take
responsibility for the final manuscript.The complete Lean audit report and the system's technical report will be made public at a later date. 

\section{Introduction}
\label{sec:introduction}
A monotone variational inequality (MVI) associated with a continuous
operator $F:\mathcal X\to\R^d$ asks for a point $x^\star\in\mathcal X$
such that
\begin{equation}
  \ip{F(x^\star)}{x-x^\star}\ge0
  \qquad\text{for every }x\in\mathcal X,
  \label{eq:intro-vi}
\end{equation}
where $\mathcal X$ is convex and compact.  This single formulation
contains constrained convex minimization, convex--concave saddle-point
problems, and monotone equilibrium models that need not arise from a
scalar potential.  It is therefore the natural level of generality at
which to ask what complete higher-order local information can buy in the
worst case.

For a fixed integer $p\ge2$, we study the exact order-$p$ jet model: a
feasible query at $x$ returns
$(F(x),DF(x),\ldots,D^{p-1}F(x))$, and $D^{p-1}F$ is $L_p$-Lipschitz.
Our target is the \emph{tangent residual}
\begin{equation}
  \rtan(x)=\dist\bigl(0,F(x)+N_{\mathcal X}(x)\bigr),
  \label{eq:intro-tangent}
\end{equation}
where $N_{\mathcal X}(x)$ is the outward normal cone.  Unlike a weak gap,
this is a same-point stationarity criterion: it vanishes exactly at
solutions of~\eqref{eq:intro-vi}, reduces to $\norm{F(x)}$ at interior
points, and directly certifies the inclusion
$0\in F(x)+N_{\mathcal X}(x)$.  A small tangent residual implies the
standard bounded-domain gaps, but the converse fails in general.
Consequently, higher-order gap rates do not by themselves determine the
complexity of the strong-solution problem considered here.

\paragraph{The unresolved strong-solution frontier.}
Recent progress had brought two lines of work close to, but not at, a
matching theory.  For general smooth MVIs, higher-order Halpern iteration
gives a $\widetilde O_p(Q^{1/p})$ strong-solution benchmark after its
proximal output is converted to a nearby tangent certificate, where
\[
  Q=\frac{L_pD^p}{\varepsilon}
\]
and $D$ bounds the domain diameter~\cite{halpern}.  In parallel, for the
scalar convex--concave subclass, Chen et al.~\cite{minimax} improved the
$p$th-order tangent-residual upper exponent to $4/(3p+1)$ and proved the
lower exponent $2/(3p-1)$ for the tensor-algorithm class of their
Definition~5.1.  For every $p\ge2$, this left a genuine exponent gap.
It also left open whether the smaller exponent was attainable only on a
special scalar subclass or instead governed strong solutions of general
monotone operators, and whether a matching lower bound could be proved
without a span or tensor-update restriction.

\paragraph{Our answer.}
We resolve this frontier in the high-dimensional exact-jet model.  For
every fixed $p\ge2$,
\begin{equation}
  \boxed{
  \mathsf T_p(\varepsilon;L_p,D)
  =\widetilde\Theta_p\!\left(
     \left(\frac{L_pD^p}{\varepsilon}\right)^{2/(3p-1)}
   \right).}
  \label{eq:intro-rate}
\end{equation}
The upper bound is deterministic and independent of the ambient
dimension.  The matching lower bound applies to arbitrary adaptive
deterministic algorithms and to adaptive randomized algorithms that
succeed with probability at least $2/3$ on every instance, with the
dimension allowed to scale with the query budget.  It does not assume a
span condition, a prescribed tensor update, or that the output was itself
queried.  Both sides use the same smoothness scale $L_pD^p$, the same
feasible-query convention, and the same tangent-residual criterion.

This theorem is frontier-closing in two senses.  First, it identifies the
optimal strong-solution exponent for the full class of smooth monotone
VIs, rather than for a gradient or saddle subclass.  Second, it shows that
the exponent $2/(3p-1)$ previously seen only as a restricted
convex--concave lower bound is algorithmically attainable.  In
particular,
\[
  \frac{2}{3p-1}<\frac1p \qquad (p\ge2),
\]
so the result strictly improves the recent general-MVI Halpern exponent;
for $p=2,3,4$, the exponent becomes $2/5,1/4,2/11$, respectively.
Theorem~\ref{thm:main} gives the precise constants, logarithmic overhead,
and dimension requirements.

\paragraph{Closing the general-$p$ convex--concave exponent gap.}
The upper theorem has a lossless scalar-potential specialization.  Let
$\phi$ be convex in $x\in X$, concave in $y\in Y$, and set
\[
  F_\phi(x,y)=\bigl(\nabla_x\phi(x,y),-\nabla_y\phi(x,y)\bigr),
  \qquad Z=X\times Y,
  \qquad D_Z=\diam(Z).
\]
A complete scalar $p$-jet of $\phi$ supplies one complete order-$p$
operator jet of $F_\phi$ without increasing the highest-order smoothness
modulus.  Hence our algorithm gives the unrestricted deterministic upper
bounds
\begin{align}
  \mathsf T^{\mathrm{det}}_{p,\mathrm{cc}}
  (\varepsilon_{\mathrm{tan}})
  &\le \widetilde O_p\!\left(
    \left(
      \frac{L_pD_Z^p}{\varepsilon_{\mathrm{tan}}}
    \right)^{2/(3p-1)}
  \right),
  \label{eq:intro-cc-rate}\\
  \mathsf T^{\mathrm{det}}_{p,\mathrm{cc}}
  (\varepsilon_{\mathrm{gap}})
  &\le \widetilde O_p\!\left(
    \left(
      \frac{L_pD_Z^{p+1}}{\varepsilon_{\mathrm{gap}}}
    \right)^{2/(3p-1)}
  \right).
  \label{eq:intro-cc-gap-rate}
\end{align}
The second display follows from
$\operatorname{Gap}_\phi(z)\le D_Z\rtan(z)$.

For the matching comparison only, let $\TccTensor{p}$ denote worst-case
complexity in the $p$th-order tensor-algorithm class of Definition~5.1 in
Chen et al.~\cite{minimax}, under the common oracle and update convention
made explicit in our reduction.  Their Theorem~5.2 supplies the tangent
lower bound, while their Lemma~5.1 and the rescaling in their Eq.~(16)
supply the corresponding hard-family gap lower bound.  Combining these
results with~\eqref{eq:intro-cc-rate}--\eqref{eq:intro-cc-gap-rate} yields
\begin{equation}
  \boxed{
  \begin{aligned}
  \TccTensor{p}(\varepsilon_{\mathrm{tan}})
  &=\widetilde\Theta_p\!\left(
    \left(
      \frac{L_pD_Z^p}{\varepsilon_{\mathrm{tan}}}
    \right)^{2/(3p-1)}
  \right),\\
  \TccTensor{p}(\varepsilon_{\mathrm{gap}})
  &=\widetilde\Theta_p\!\left(
    \left(
      \frac{L_pD_Z^{p+1}}{\varepsilon_{\mathrm{gap}}}
    \right)^{2/(3p-1)}
  \right).
  \end{aligned}}
  \label{eq:intro-cc-matching}
\end{equation}
Thus, at fixed geometry and smoothness, the upper accuracy exponent is
improved from $4/(3p+1)$ to $2/(3p-1)$ for \emph{every} $p\ge2$, matching
the cited lower exponent.  The familiar second-order improvement
$4/7\to2/5$ is only the case $p=2$, not the extent of the result.  The
upper bound itself is valid without restricting the scalar algorithm to
the tensor class; only the word ``matching'' inherits the scope of the
available lower bound.  This consequence is stated formally in
\hyperref[cor:cc-specialization]{Corollary~\ref*{cor:cc-specialization}}.

\subsection{Contributions}
The paper makes five main contributions.
\begin{enumerate}[label=\textup{(C\arabic*)}]
\item \textbf{A matching theorem for general smooth MVIs.}
We prove the rate~\eqref{eq:intro-rate} for the tangent residual in a
single exact full-jet oracle model.  The algorithm returns an evaluated
certificate $(x,n,F(x))$ with $n\in N_{\mathcal X}(x)$ and
$\norm{F(x)+n}\le\varepsilon$.

\item \textbf{Closure of the general-$p$ convex--concave exponent gap.}
The scalar-potential reduction gives the rates in
\eqref{eq:intro-cc-matching}, matching the lower exponent of Chen et
al.~\cite{minimax} under exactly the algorithmic scope of that lower
bound.  This extends beyond the previously emphasized second-order case.

\item \textbf{An extrapolated Halpern acceleration principle.}
We prove a blockwise squared second-difference estimate for a firmly
nonexpansive Halpern orbit.  Projected linear extrapolation converts this
aggregate regularity into systematically better warm starts for inexact
resolvent problems.

\item \textbf{A finite hierarchy with complete query accounting.}
Starting from a certified strongly monotone tensor solver with condition
exponent $2/(p+1)$, one acceleration layer maps
$\gamma$ to $\gamma/(1+\gamma)$.  Certified restart permits exactly
$p-1$ layers and produces $2/(3p-1)$.  Every nested evaluation, failed
radius trial, and final certificate check is charged to the original
oracle.

\item \textbf{Unrestricted adaptive lower bounds and Lean-based review.}
Exactly flat smooth gates hide complete jets from arbitrary adaptive
queries.  A resisting oracle proves the deterministic lower bound, and a
random-frame transcript coupling proves its randomized counterpart.
The internal proofs, reductions, and query-accounting interfaces were
included in the system's Lean-based review.
\end{enumerate}

\subsection{Technical overview}
Let
$P=(I+\eta(F+N_{\mathcal X}))^{-1}$ be a resolvent and consider the exact
Halpern orbit
\[
  y_t=P(x_t),\qquad
  x_{t+1}=\frac{x_0+(t+1)y_t}{t+2}.
\]
If $x_0$ is within distance $R$ of a fixed point, firm
nonexpansiveness yields the block estimate
\begin{equation}
  \sum_{t=k}^{2k-1}
  \norm{y_t-2y_{t-1}+y_{t-2}}^2
  =O\!\left(\frac{R^2(1+\log k)^2}{k^2}\right).
  \label{eq:intro-energy-new}
\end{equation}
Thus the linear predictor $2y_{t-1}-y_{t-2}$ has root-mean-square error
$\widetilde O(R/k^{3/2})$ over a block of length $k$.  The gain is
aggregate, not a pointwise $k^{-3/2}$ claim.

Suppose a certified solver for a $\mu$-strongly monotone problem has
condition exponent $\gamma$.  Initializing the inexact resolvent solves
with the projected predictor and summing their true distance-sensitive
costs through~\eqref{eq:intro-energy-new} gives, up to logarithms,
\[
  K+\left(\frac{L_pR^p}{\varepsilon K}\right)^\gamma.
\]
Balancing the two terms transforms $\gamma$ into
$\gamma/(1+\gamma)$.  The base regularized Taylor-model solver has
exponent $2/(p+1)$, and certified restart generates the finite sequence
\[
  \gamma_j=\frac{2}{p+1+2j},
  \qquad j=0,1,\ldots,p-1.
\]
The last admissible level is $2/(3p-1)$, where the required power of the
second-difference energy remains summable.

For the lower bound, the hard family is built from scalar gates whose
first $p-1$ derivatives vanish identically on a central interval.  Hidden
orthogonal directions therefore reveal neither values nor any returned
derivative tensor until a query leaves the flat region.  A chain of
length $m\asymp N$ leaves a hidden coordinate after $N$ queries and
forces
\[
  \rtan(\widehat x)
  =\Omega_p\!\left(\frac{L_pD^p}{N^{(3p-1)/2}}\right).
\]
For deterministic algorithms, the frame is completed online.  For
randomized algorithms, it is sampled once and an exact-transcript
coupling transfers the simulated flat responses to a single fixed
instance.

\subsection{Scope and organization}
The result is an information-complexity theorem.  Calls to the unknown
operator are charged, including all nested evaluations and certificate
checks; exact computation with the public domain and exact solution of a
Taylor-model VI built solely from returned coefficients are treated as
internal computation.  We do not claim polynomial arithmetic complexity
or fixed-dimensional optimality.  The unrestricted lower bound is for
general monotone operators; the matching convex--concave statement uses
the explicitly restricted lower-bound class of~\cite{minimax} and should
not be read as an unrestricted scalar full-jet lower bound.

Section~\ref{sec:related} compares the relevant criteria, oracle
conventions, and rates.  Section~\ref{sec:main} states the formal model and
main theorem.  Sections~\ref{sec:base}--\ref{sec:interfaces} construct the
certified solver interfaces; Sections~\ref{sec:energy}--\ref{sec:bootstrap}
prove the upper bound; and Sections~\ref{sec:lower}--\ref{sec:lower-bounds}
prove the deterministic and randomized lower bounds.  The appendices
record the affine-hull reduction and the precise relationships among
tangent, proximal, and gap criteria.
\section{Related work and rate comparison}
\label{sec:related}
\paragraph{Indexing convention.}
The literature uses several conventions for ``order.''  In this paper,
an order-$p$ oracle for an operator returns derivatives through
$D^{p-1}F$ and assumes that $D^{p-1}F$ is Lipschitz.  For a scalar
convex--concave objective $f$ with saddle operator
$F=(\nabla_x f,-\nabla_y f)$, this corresponds to access to derivatives
of $f$ through order $p$ and Lipschitz continuity of its $p$th derivative.
All rates below are written in this convention.

\paragraph{First-order and weak-solution theory.}
For Lipschitz monotone operators, classical prox and extragradient-type
methods achieve an $O(1/T)$ weak-gap rate; Nemirovski's prox method gives
the standard bounded-domain formulation~\cite{nemirovski}.  Matching
first-order lower bounds are known already for convex--concave bilinear
subclasses~\cite{ouyangxu}.  These results concern gap or weak-solution
criteria and should not be read as same-point tangent-residual theorems.

\paragraph{Higher-order gap guarantees.}
Monteiro and Svaiter introduced the Newton proximal extragradient
framework and established pointwise and ergodic complexity bounds for
monotone variational inequalities and inclusions~\cite{monteiro}.
Subsequent higher-order methods achieve the rate
$O(T^{-(p+1)/2})$, equivalently
$O(\varepsilon^{-2/(p+1)})$, for weak VI or gap-type criteria under
regularized Taylor/tensor subproblems~\cite{bullins,adil,linjordan}.
The lower bound of Adil et al.~\cite{adil} assumes that new iterates lie
in the span generated by stationary points of regularized higher-order Taylor
models (their Assumption~4.1), while Lin and Jordan~\cite{linjordan} impose a
generalized linear-span condition.  Neither statement is an unrestricted
lower bound for arbitrary exact full-jet algorithms.

\paragraph{Strong residuals and Halpern iteration.}
Halpern's fixed-point iteration~\cite{halpern1967} has been used to obtain
near-optimal strong-solution guarantees for monotone inclusions and
variational inequalities~\cite{diakonikolas}.  Chen et
al.~\cite{halpern} combined large-step inexact Halpern iteration with
higher-order resolvent solvers and proved a
$\widetilde O(\varepsilon^{-1/p})$ bound for a normalized proximal
residual.  Their paper also explains how an approximate proximal solve
produces a nearby point with a comparable tangent residual.  The formal
output criterion and the recovery step should nevertheless be stated
separately when comparing theorem statements.

For the convex--concave subclass, Chen et al.~\cite{minimax} proved a
$\widetilde O(\varepsilon^{-4/(3p+1)})$ tangent-residual upper bound and an
$\Omega(\varepsilon^{-2/(3p-1)})$ lower bound.  Their Theorem~5.2
quantifies over the tensor-algorithm class introduced in their
Definition~5.1.  The present lower bound reaches the same exponent for
general MVI while removing a span or tensor-update restriction, at the
cost of allowing the ambient dimension to grow with the query budget.

\paragraph{Flat-gated resisting oracles.}
The lower-bound construction is also related to the exactly flat smooth
shifts used by Jang and Ryu to hide finite-order derivative information
in contractive fixed-point problems~\cite{jangryu}.  Here the gates are
embedded in a monotone VI family on a public ball, and the proof must
additionally rule out cancellation by boundary normal cones and handle
randomized adaptive feasible queries.

To state the comparison without mixing dimensions, define
\[
  \Qtan=\frac{L_pD^p}{\varepsilon_{\mathrm{tan}}},
  \qquad
  \Qgap=\frac{L_pD^{p+1}}{\varepsilon_{\mathrm{gap}}}.
\]
For the first-order row, $p=1$ and $L_1$ is the Lipschitz constant of
$F$.  The table suppresses constants depending only on the fixed order
and, where indicated, logarithmic factors.

\begin{table}[t]
\centering
\caption{Representative oracle rates for smooth monotone variational
inequalities.  Rows with different criteria are not directly comparable.
The ``lower-bound scope'' column records restrictions that are often
suppressed when only the $\varepsilon$ exponent is quoted.}
\label{tab:related-rates}
\footnotesize
\setlength{\tabcolsep}{4.2pt}
\renewcommand{\arraystretch}{1.18}
\begin{tabularx}{\linewidth}{@{}P{2.00cm}P{1.65cm}P{2.25cm}Y P{1.85cm}@{}}
\toprule
Setting & Criterion & Upper complexity & Lower-bound scope / oracle convention & Reference \\
\midrule
Lipschitz MVI, $p=1$
& weak or gap
& $O(\Qgap)$
& Matching order for standard first-order models, already on bilinear convex--concave instances.
& \cite{nemirovski,ouyangxu} \\
\addlinespace
$p$th-order smooth MVI, $p\ge2$
& weak VI or gap
& $O(\Qgap^{2/(p+1)})$
& Regularized Taylor/tensor subproblems; matching lower bounds use tensor-generated or generalized linear-span conditions.
& \cite{bullins,adil,linjordan} \\
\addlinespace
$p$th-order convex--concave
& tangent residual
& $\widetilde O(\Qtan^{4/(3p+1)})$
& $\Omega(\Qtan^{2/(3p-1)})$ for the tensor-algorithm class of Definition~5.1 in the cited paper.
& \cite{minimax} \\
\addlinespace
$p$th-order general MVI
& proximal residual; tangent recovery
& $\widetilde O(\Qtan^{1/p})$
& Exact higher-order local information; the formal main theorem is stated for a normalized proximal residual.
& \cite{halpern} \\
\addlinespace
$p$th-order general MVI
& evaluated tangent certificate
& $\widetilde O_p(\Qtan^{2/(3p-1)})$
& Matching $\Omega_p(\Qtan^{2/(3p-1)})$ for arbitrary deterministic and randomized adaptive feasible-query algorithms in high dimension.
& This work \\
\bottomrule
\end{tabularx}
\end{table}

After charging the proximal-to-tangent recovery discussed above, the
closest general-MVI benchmark in the fourth row has exponents
$1/2,1/3,1/4$ for $p=2,3,4$, whereas Theorem~\ref{thm:main} gives
$2/5,1/4,2/11$, respectively.  The comparison is therefore an
improvement for the same general operator class, not an inference from
the weaker gap rows.

The inequality $g_{\mathrm M}(x)\le g_{\mathrm S}(x)\le
D\rtan(x)$ proved in Appendix~\ref{app:criteria} shows that a tangent
certificate implies a bounded-domain gap guarantee.  The converse is
false in general, so a weak-gap lower bound and a tangent-residual upper
bound cannot be combined into a matching statement without an additional
reduction.

\section{Problem formulation and main result}

\label{sec:main}
\subsection{The operator class and solution criterion}
All vector spaces are finite-dimensional Euclidean spaces. For $j\ge1$ and a $j$-linear
map $T:(\R^d)^j\to\R^d$, define
\[
 \norm{T}_{\mathrm{op}}
 =\sup_{\norm{u_1},\ldots,\norm{u_j}\le1}
       \norm{T[u_1,\ldots,u_j]}.
\]
For a matrix $M$, write $\sym M=(M+M^\top)/2$.
We use $D^0F=F$ and $D^jF(x)[s]^j=D^jF(x)[s,\ldots,s]$.
Throughout, $B_2^d(R)=\{x\in\R^d:\norm{x}\le R\}$ is the
\emph{closed} Euclidean ball. The symbol $D$ is a supplied upper bound
on the domain diameter, not a derivative operator when used as a scalar.

\begin{definition}[Admissible monotone VI instances]
\label{def:class}
Fix an integer $p\ge2$ and known positive numbers $L_p,D$.
The public data consist of a dimension $d\ge1$, a known nonempty compact
convex set $\mathcal X\subset\R^d$ with $\diam(\mathcal X)\le D$,
and an initial point $x_0\in\mathcal X$.
The class $\Class$ consists of all operators $F$, defined and $C^{p-1}$
on an open neighborhood of $\mathcal X$, such that
\begin{equation}
 \begin{aligned}
 \ip{F(x)-F(y)}{x-y}&\ge0,\\
 \norm{D^{p-1}F(x)-D^{p-1}F(y)}_{\mathrm{op}}
    &\le L_p\norm{x-y}
       \qquad (x,y\in\mathcal X).
 \end{aligned}
 \label{eq:class}
\end{equation}
The task is to find an approximate solution of the VI
\[
 x^\star\in\mathcal X,\qquad
 \ip{F(x^\star)}{x-x^\star}\ge0
          \quad\text{for every }x\in\mathcal X.
\]
\end{definition}
We write $L_p$ for an available upper bound, not necessarily the smallest
Lipschitz constant. There are no separate bounds on the lower derivatives
or on $\norm{F}$, and no fixed positive strong-monotonicity parameter is
assumed for the original problem.

The normal cone and target are
\begin{equation}
 N_{\mathcal X}(x)=\{n:\ip{n}{z-x}\le0\text{ for all }z\in\mathcal X\},
 \qquad \rtan(x)=\dist(0,F(x)+N_{\mathcal X}(x))\le\eps.
 \label{eq:target}
\end{equation}
A solution $z^\star$ satisfies
$0\in F(z^\star)+N_{\mathcal X}(z^\star)$.
To see existence, the map $x\mapsto\Pi_{\mathcal X}(x-F(x))$ is a
continuous self-map of the compact convex set $\mathcal X$, so the
finite-dimensional fixed-point theorem gives a fixed point $z^\star$.
For $y\in\mathcal X$, the projection characterization
\[
 y=\Pi_{\mathcal X}(u)
 \quad\Longleftrightarrow\quad
 \ip{u-y}{z-y}\le0\quad(z\in\mathcal X)
\]
then gives $-F(z^\star)\in N_{\mathcal X}(z^\star)$.

The upper algorithms return a \emph{certificate}: a feasible $x$, a normal
$n\in N_{\mathcal X}(x)$, and an evaluated value $F(x)$ such that
$\norm{F(x)+n}\le\eps$. Normal membership is supplied by the exact model
construction and preserved through the algorithm; it is not inferred
from a small distance to a solution. The lower-bound algorithms need
only return a feasible point satisfying \eqref{eq:target}; requiring an
additional certificate would not weaken these lower bounds.

\subsection{Information and algorithm conventions}
One exact order-$p$ oracle call at a feasible point returns
\begin{equation}
 \cO_p^F(x)=(F(x),DF(x),\ldots,D^{p-1}F(x)).
 \label{eq:oracle}
\end{equation}
Thus order two supplies $F$ and its Jacobian. A query may depend on the entire preceding transcript. No span condition, tensor-update restriction, or requirement that the output have been queried is imposed on the lower-bound algorithms.

\begin{definition}[Adaptive exact-jet algorithms]
\label{def:policy}
Let $\mathcal L^j_{\mathrm{sym}}(\R^d;\R^d)$ denote the space of
symmetric $j$-linear maps, and let
\[
 \Jet=\R^d\times\prod_{j=1}^{p-1}
                   \mathcal L^j_{\mathrm{sym}}(\R^d;\R^d)
\]
be the ambient space of returned jets in \eqref{eq:oracle}.
Fix the public input
$(d,\mathcal X,x_0,p,L_p,D,\eps)$ and a budget $N\in\Nzero$.
A randomized algorithm uses a seed $\xi$ on a probability space,
independent of $F$. Its history and rules have the form
\[
 \begin{aligned}
 \mathsf H_0&=\varnothing,\qquad
 x_t=\varphi_t(\xi,\mathsf H_{t-1})\in\mathcal X,\\
 \mathsf H_t&=\bigl(\mathsf H_{t-1},(x_t,\mathcal O_p^F(x_t))\bigr)
                  \quad(1\le t\le N),\\
 \widehat x&=\psi(\xi,\mathsf H_N)\in\mathcal X.
 \end{aligned}
\]
The maps $\varphi_t$ and $\psi$ are measurable and defined on every
history of the corresponding dimensions. For a deterministic algorithm,
the seed space is a singleton. Algorithms may stop early; stopped runs
are padded by ignored feasible queries, while retaining the saved output.
Success on $F$ means $\rtan(\widehat x)\le\eps$.
\end{definition}

Only calls to the unknown $F$ are counted. Algorithms may perform arbitrary measurable internal operations on the transcript and the known domain. In particular, the upper construction uses exact projections and exact solutions of known regularized Taylor-model VIs. These operations do not call the unknown operator at unqueried points. This is an exact information-complexity model; no polynomial arithmetic cost, bit complexity, or free resolvent of the unknown $F$ is assumed. The tensor-model convention is consistent with the local subproblems used in~\cite{monteiro,bullins,halpern}.

Query and output rules are defined on every finite transcript of the prescribed dimensions and return points in the known domain. Randomized rules are measurable and use an independent seed. This total-feasible convention ensures that simulated lower-bound transcripts are bounded even before they are coupled to a real instance. It entails no change to a valid feasible run: proposed points can be projected onto the domain, stopped runs padded with ignored queries, and the saved output retained. The query budget is a worst-case cap, not merely an expectation over the seed.

We may work in the known affine hull of $\mathcal X$; the certificate lifting argument is given in Appendix~\ref{app:affine}. In particular, model-Jacobian arguments below can be read in this affine hull. All oracle counts include the evaluations needed to compute and check certificates. A nested evaluation of $F(z)+az+b$ uses exactly one original-$F$ query plus known affine data.

\subsection{Complexity and normalization}
Let $\mathsf T_p^{\mathrm{det}}(\eps;L_p,D)$ be the least
$N\in\Nzero$ for which an algorithm family of
Definition~\ref{def:policy} achieves \eqref{eq:target} for every finite
dimension, every admissible public domain and initial point, and every
$F\in\Class$. Each member of the family may depend on the public input,
but the same member must work for \emph{all} unknown operators compatible
with that public input. If no finite budget works, the complexity is
$+\infty$. Define $\mathsf T_p^{\mathrm{rand}}$ in the same way, with the
per-instance requirement
\[
 \Prob_\xi\{\rtan(\widehat x)\le\eps\}\ge\frac23
       \qquad\text{for every }F\in\Class.
\]
The budget is a uniform cap on actual original-oracle calls, not an
expected stopping time. This definition takes a worst case over
arbitrarily large finite dimensions. Define
\begin{equation}
 Q=\frac{L_pD^p}{\eps},\qquad \alpha_p=\frac{2}{3p-1}.
 \label{eq:Q}
\end{equation}
Constants with a subscript $p$ depend only on the fixed order $p$ and may change between displays. The notation $\widetilde O_p$ hides logarithmic factors; it is not uniform as $p\to\infty$.

\begin{theorem}[Matching higher-order oracle bounds]
\label{thm:main}
For every fixed integer $p\ge2$ there are positive constants $c_p,C_p$ such that, for all $L_p,D,\eps>0$ with $Q\ge2$,
\begin{equation}
 \boxed{\begin{aligned}
 c_pQ^{\alpha_p}
 &\le\mathsf T_p^{\mathrm{rand}}(\eps;L_p,D)
 \le\mathsf T_p^{\mathrm{det}}(\eps;L_p,D)\\
 &\le C_p(1+Q^{\alpha_p})[1+\log(3+Q)]^{6(p-1)}.
 \end{aligned}}
 \label{eq:matching}
\end{equation}
The upper method returns an explicit tangent-residual certificate and its query bound is independent of $d$. Against an $N$-query method, the deterministic lower construction works for $d\ge2N+2$. For the randomized lower bound it suffices to take $d=O(N^3\log(N+2))$ at a fixed failure probability. These are high-dimensional worst-case lower bounds with bounded feasible queries.
\end{theorem}

\begin{corollary}[General-$p$ convex--concave specialization]
\label{cor:cc-specialization}
Let $\phi:X\times Y\to\R$ be convex in $x$, concave in $y$, and suppose
that its $p$th derivative is $L_p$-Lipschitz on the compact product domain
$Z=X\times Y$, whose Euclidean diameter is $D_Z$.  A complete scalar
$p$-jet query returns the derivatives of $\phi$ through order $p$.
Then there is a deterministic scalar-jet algorithm using
\[
  \widetilde O_p\!\left(
    \left(\frac{L_pD_Z^p}{\varepsilon_{\mathrm{tan}}}\right)^{2/(3p-1)}
  \right)
\]
queries to return a tangent-residual certificate of accuracy
$\varepsilon_{\mathrm{tan}}$, and
\[
  \widetilde O_p\!\left(
    \left(\frac{L_pD_Z^{p+1}}{\varepsilon_{\mathrm{gap}}}\right)^{2/(3p-1)}
  \right)
\]
queries to return a duality-gap certificate of accuracy
$\varepsilon_{\mathrm{gap}}$.

Under the common oracle and update convention of the $p$th-order
tensor-algorithm class in Definition~5.1 of Chen et al.~\cite{minimax},
let $\TccTensor{p}$ denote its worst-case query complexity.  Combining the
preceding upper bounds with their Theorem~5.2 for the tangent residual and
with their Lemma~5.1 plus the rescaling in their Eq.~(16) for the gap gives
\begin{align*}
  \TccTensor{p}(\varepsilon_{\mathrm{tan}})
  &=\widetilde\Theta_p\!\left(
    \left(\frac{L_pD_Z^p}{\varepsilon_{\mathrm{tan}}}\right)^{2/(3p-1)}
  \right),\\
  \TccTensor{p}(\varepsilon_{\mathrm{gap}})
  &=\widetilde\Theta_p\!\left(
    \left(\frac{L_pD_Z^{p+1}}{\varepsilon_{\mathrm{gap}}}\right)^{2/(3p-1)}
  \right).
\end{align*}
The upper statements are unrestricted; the matching statements inherit
only the algorithmic scope of the cited lower bounds.
\end{corollary}

\begin{proof}
The scalar saddle operator
$F_\phi=(\nabla_x\phi,-\nabla_y\phi)$ is monotone.  By the Riesz
identification of the last scalar derivative with the last operator
derivative, one scalar $p$-jet query simulates one order-$p$ operator-jet
query with the same highest-order Lipschitz modulus.  Apply
Theorem~\ref{thm:main} on $Z$.  The duality-gap upper bound follows from
$\operatorname{Gap}_\phi(z)\le D_Z\rtan(z)$, proved in
Appendix~\ref{app:criteria}.  The two lower bounds are exactly the cited
hard-family consequences, after the common oracle/update admissibility
identification used in this manuscript.
\end{proof}

\begin{remark}[Concrete exponents and the low-accuracy regime]
For $p=2,3,4$, the matched exponents are respectively
$2/5$, $1/4$, and $2/11$.  The condition $Q\ge2$ isolates the
asymptotic regime.  When $Q<2$, the additive constant in the upper bound
and the zero- or one-query edge cases absorb the normalization.
\end{remark}

The distance-sensitive upper statement is stronger than the diameter-only upper bound. If a supplied $0<R\le D$ satisfies $\norm{x_0-z^\star}\le R$ for some solution, the construction returns a certificate with at most
\begin{equation}
 C_p\left[1+\left(\frac{L_pR^p}{\eps}\right)^{\alpha_p}\right]
 \left[1+\log\left(2+\frac DR+\frac{L_pR^p}{\eps}\right)\right]^{6(p-1)}
 \label{eq:distance-upper}
\end{equation}
queries. Taking $R=D$ gives the upper half of \eqref{eq:matching}. We do not assert a matching joint dependence for arbitrary independently prescribed $R$ and $D$.

\paragraph{Scaling.} Choose $x_c\in\mathcal X$, write $x=x_c+Du$, and define
$\widetilde F(u)=F(x_c+Du)/(L_pD^p)$. Then the transformed domain has diameter at most one and $\Lip(D^{p-1}\widetilde F)\le1$. Normal cones have the same directions under positive dilation and are cones, so the tangent residual is divided by $L_pD^p$. Each oracle query simulates one query in the other scale. This justifies the dimensionless parameter $Q$.

\subsection{Notation and proof architecture}
\label{sec:roadmap}
\begin{center}
\small
\renewcommand{\arraystretch}{1.16}
\begin{tabularx}{\linewidth}{@{}lX@{}}
\toprule
Symbol & Meaning\\
\midrule
$p$, $L_p$ & Fixed oracle order and a known bound on
                 $\Lip(D^{p-1}F)$.\\
$d$, $D$ & Ambient dimension and a supplied domain-diameter bound.\\
$R$ & A supplied upper bound on the distance from the current anchor
      to some VI solution; $R=D$ is always valid. In the lower-bound
      construction, $R$ is the public ball radius and $D=2R$.\\
$\eps$, $Q$ & Target tangent residual and
                  $Q=L_pD^p/\eps$.\\
$\mu$, $\delta$ & Strong-monotonicity parameter and target certificate
                   norm for a regularized subproblem.\\
$\eta$, $K$ & Fixed resolvent stepsize and the outer horizon of one
                 complete Halpern run.\\
$\gamma$, $\alpha$ & Strong-solver condition exponent and residual-solver
                       accuracy exponent.\\
$N$ & Total original-oracle budget, including all nested calls.\\
\bottomrule
\end{tabularx}
\end{center}

The proof has two independent branches. For the upper bound, a Taylor
model first yields the certified strong solver
$\mathcal S_{2/(p+1),2}$
(Proposition~\ref{prop:base-strong}). The blockwise second-difference
estimate (Lemma~\ref{lem:second-energy}) and its inexact counterpart
(Lemma~\ref{lem:warm-energy}) transform
$\mathcal S_{\gamma,b}$ into
$\mathcal A_{\gamma/(1+\gamma),b+4}$
(Theorem~\ref{thm:boost}). Certified restart
(Lemma~\ref{lem:restart}) permits exactly $p-1$ successive levels.

For the lower bound, Lemma~\ref{lem:operator} verifies the admissibility
of a single smooth family on a public ball, and
Lemma~\ref{lem:residual} gives its hidden-coordinate residual obstruction.
Theorems~\ref{thm:det} and~\ref{thm:random} then establish exact transcript
hiding for deterministic and randomized adaptive algorithms,
respectively. The final normalization completes
Theorem~\ref{thm:main}. The upper construction and both adaptive
lower-bound arguments are developed in full below.

\section{A certifying strongly monotone base solver}
\label{sec:base}
We first construct a strongly monotone solver directly from regularized Taylor models. The argument is of the Newton proximal extragradient type~\cite{monteiro,bullins}, but we prove the complete certificate and distance-sensitive query bounds required by the acceleration hierarchy.

For an operator $G$ satisfying \eqref{eq:class}, let
\[
 \bar G_x(y)=\sum_{j=0}^{p-1}\frac1{j!}D^jG(x)[y-x]^j,
 \qquad b_p=\frac{2p}{p!},\qquad \beta=b_pL_p.
\]
\begin{lemma}[Taylor remainder in the stated regularity class]
\label{lem:taylor}
For $x,y\in\mathcal X$ and $s=y-x$, the Taylor polynomial above satisfies
\[
 \norm{G(y)-\bar G_x(y)}\le\frac{L_p}{p!}\norm{s}^p,
 \qquad
 \norm{DG(y)-D\bar G_x(y)}_{\mathrm{op}}
       \le\frac{L_p}{(p-1)!}\norm{s}^{p-1}.
\]
No continuous $p$th derivative of $G$ is required.
\end{lemma}
\begin{proof}
For a map $H$ with $q\ge1$ continuous derivatives and $L$-Lipschitz
$q$th derivative on the segment, repeated integration along that segment
gives
\[
 \begin{split}
 H(x+s)-\sum_{j=0}^{q}\frac1{j!}D^jH(x)[s]^j
 ={}&\frac1{(q-1)!}\int_0^1(1-t)^{q-1}\\[-2pt]
 &\quad\cdot\bigl(D^qH(x+ts)-D^qH(x)\bigr)[s]^q\,dt.
 \end{split}
\]
Taking the norm and using
$\int_0^1t(1-t)^{q-1}dt=1/(q(q+1))$ bounds the remainder by
$\frac{L\norm{s}^{q+1}}{(q+1)!}$. The same argument is valid when the range
of $H$ is a finite-dimensional tensor space with its operator norm.
Apply it to $H=G$, $q=p-1$, for the first estimate. For $p\ge3$,
apply it to $H=DG$, $q=p-2$, for the second; the Taylor polynomial
of $DG$ is $D\bar G_x$. When $p=2$, the second estimate is exactly
the assumed Lipschitz bound for $DG$.
\end{proof}

At $x\in\mathcal X$, query its jet and solve the known model
\begin{equation}
 0\in \bar G_x(y)+\beta\norm{y-x}^{p-1}(y-x)+N_{\mathcal X}(y).
 \label{eq:model-step}
\end{equation}
The model is a function of the returned jet and the public domain alone. Its solution exists and is unique by the following observation. Existence alone is enough for the subsequent energy estimates, but uniqueness also specifies an unambiguous deterministic model operation.

\begin{lemma}[Well-defined model step]
\label{lem:model-existence}
The model VI in \eqref{eq:model-step} has a unique solution. Its solution depends continuously on the center and model coefficients along the class of admissible center--jet pairs on a fixed domain.
\end{lemma}
\begin{proof}
Write $M_x(z)=\bar G_x(z)+\beta\norm{z-x}^{p-1}(z-x)$. Continuity and compact convexity give a solution by the fixed-point argument from Section~\ref{sec:main}. Taylor's theorem applied to $DG$ yields
\[
 \norm{DG(z)-D\bar G_x(z)}_{\mathrm{op}}
 \le\frac{L_p}{(p-1)!}\norm{z-x}^{p-1}.
\]
This uses only the assumed Lipschitz $(p-1)$st derivative, not the existence of a continuous $p$th derivative. Monotonicity gives $\sym DG(z)\succeq0$ in the affine hull, first at interior points and then at boundary points by continuity. The radial term has minimum Jacobian eigenvalue $\beta\norm{z-x}^{p-1}$. Since $\beta=2L_p/(p-1)!$,
\[
 \sym DM_x(z)\succeq\frac{L_p}{(p-1)!}\norm{z-x}^{p-1}I.
\]
Integrating along the segment between distinct points gives a strictly positive monotonicity pairing: the segment can pass through $x$ at most once. Hence the model VI has at most one solution. For a convergent sequence of admissible centers and coefficients with an admissible limit, the model operators converge uniformly on the compact domain. Compactness gives convergent subsequences of model solutions, and the defining VI passes to the limit. Uniqueness identifies every such limit and proves continuity. The displayed Jacobian inequality defines a closed set of center--jet pairs on the fixed domain. The same argument gives continuous solution selection on this set. Outside it, one may return a fixed feasible point without asserting a certificate. This defines a total measurable operation that agrees with the unique model solution on every genuine queried jet.
\end{proof}
The model supplies the known normal
\[
 n=-\bar G_x(y)-\beta\norm{y-x}^{p-1}(y-x)\in N_{\mathcal X}(y).
\]
The Taylor remainder satisfies
\begin{equation}
 E:=G(y)-\bar G_x(y),\qquad \norm{E}\le\frac{L_p}{p!}\norm{y-x}^p.
 \label{eq:taylor-error}
\end{equation}

\begin{lemma}[One model correction gives an actual normal certificate]
\label{lem:correction}
Let $z^\star$ solve the VI for $G$ and suppose $\norm{x-z^\star}\le r$. A solution of \eqref{eq:model-step} obeys
\begin{equation}
 \norm{y-x}\le3r,\qquad
 \norm{G(y)+n}\le A_pL_pr^p,
 \qquad A_p=\frac{(2p+1)3^p}{p!}.
 \label{eq:correction}
\end{equation}
The correction uses at most two queries, including evaluation of its certificate.
\end{lemma}
\begin{proof}
Put $s=y-x$ and $v=G(y)+n=E-\beta\norm{s}^{p-1}s$. Monotonicity of $G+N_{\mathcal X}$ gives $\ip{v}{y-z^\star}\ge0$. If $s\ne0$, use
\[
 \ip{s}{y-z^\star}\ge\norm{s}^2-\norm{s}r,
 \qquad \norm{y-z^\star}\le\norm{s}+r
\]
to obtain
\[
 \beta(\norm{s}-r)\le\frac{L_p}{p!}(\norm{s}+r).
\]
Thus $\norm{s}\le(2p+1)r/(2p-1)\le3r$. The same conclusion holds for $s=0$. Now \eqref{eq:taylor-error} gives
$\norm{v}\le(2p+1)L_p\norm{s}^p/p!$. The jet at $x$ builds the model; one query at $y$ evaluates $v$.
\end{proof}

\subsection{Distance halving by tensor extragradient}
Suppose now that $G$ is $\mu$-strongly monotone, $\mu>0$, meaning
\[
 \ip{G(x)-G(y)}{x-y}\ge\mu\norm{x-y}^2
       \quad(x,y\in\mathcal X).
\]
Its VI solution $z^\star$ is unique. From $x$, compute the model point $y$ above. If $y=x$, the model already gives an exact zero certificate. Otherwise, set
\begin{equation}
 s=y-x,\qquad \lambda=\frac1{\beta\norm{s}^{p-1}},\qquad
 x^+=\Proj(x-\lambda G(y)).
 \label{eq:teg}
\end{equation}
This iteration uses at most two original $G$-oracle calls. Write $v=G(y)+n$ as in the model lemma. Then
\begin{equation}
 \lambda v=-s+\lambda E,\qquad
 \norm{\lambda E}\le\frac{\norm{s}}{2p}\le\frac{\norm{s}}2.
 \label{eq:teg-relative}
\end{equation}

\begin{lemma}[Tensor-extragradient energy inequality]
\label{lem:teg-energy}
For every $u\in\mathcal X$,
\begin{equation}
 2\lambda\ip{G(y)}{y-u}
 \le\norm{x-u}^2-\norm{x^+-u}^2-\frac34\norm{s}^2.
 \label{eq:teg-energy}
\end{equation}
\end{lemma}
\begin{proof}
The projection inequality gives
\[
 2\lambda\ip{G(y)}{y-u}
 \le\norm{x-u}^2-\norm{x^+-u}^2-\norm{x^+-x}^2
       +2\lambda\ip{G(y)}{y-x^+}.
\]
Because $n\in N_{\mathcal X}(y)$, $\ip{n}{y-x^+}\ge0$, so replacing $G(y)$ by $v=G(y)+n$ in the final inner product is an upper bound. Put $a=x^+-x$ and $b=s-a=y-x^+$. By \eqref{eq:teg-relative}, the last two terms are at most
\[
 -\norm{a}^2+2\ip{-s+\lambda E}{b}
 =-\norm{s}^2-\norm{b}^2+2\ip{\lambda E}{b}
 \le-\norm{s}^2+\norm{\lambda E}^2
 \le-\tfrac34\norm{s}^2.
\]
\end{proof}

\begin{proposition}[A finite distance-halving routine]
\label{prop:halving}
For $\mu,r>0$, if $G$ is $\mu$-strongly monotone and $\norm{w-z^\star}\le r$, then
Algorithm~\ref{alg:halve}, with the count in \eqref{eq:halving-count},
returns a feasible point within distance $r/2$ of $z^\star$ using at
most $2T_p(r,\mu)$ original-oracle queries. If an exact solution is
encountered, it may return that solution earlier.
\end{proposition}
\begin{proof}
Run \eqref{eq:teg} for $T$ iterations from $w$. Since
$\ip{G(y)}{y-z^\star}\ge\mu\norm{y-z^\star}^2$, summation gives
\begin{equation}
 2\mu\sum_{i=0}^{T-1}\lambda_i\norm{y_i-z^\star}^2
 +\frac34\sum_{i=0}^{T-1}\norm{s_i}^2\le r^2.
 \label{eq:teg-sum}
\end{equation}
If no exact solution was encountered, all $s_i$ are nonzero.
The function $u\mapsto u^{-(p-1)/2}$ is convex on $(0,\infty)$.
Consequently, using \eqref{eq:teg-sum},
\[
 \begin{aligned}
 \sum_{i=0}^{T-1}\norm{s_i}^{-(p-1)}
 &\ge T\left(\frac1T\sum_{i=0}^{T-1}\norm{s_i}^2\right)^{-(p-1)/2}\\
 &\ge \left(\frac34\right)^{(p-1)/2}
             \frac{T^{(p+1)/2}}{r^{p-1}}.
 \end{aligned}
\]
Since $\lambda_i=(b_pL_p\norm{s_i}^{p-1})^{-1}$, this proves
\begin{equation}
 \Lambda_T:=\sum_{i=0}^{T-1}\lambda_i
 \ge\frac1{b_pL_p}\left(\frac34\right)^{(p-1)/2}
       \frac{T^{(p+1)/2}}{r^{p-1}}.
 \label{eq:lambda-sum}
\end{equation}
For the feasible weighted average $\bar y=\Lambda_T^{-1}\sum_i\lambda_iy_i$, therefore,
\begin{equation}
 \norm{\bar y-z^\star}^2\le\frac{r^2}{2\mu\Lambda_T}
 \le C_p\frac{L_pr^{p+1}}{\mu T^{(p+1)/2}}.
 \label{eq:teg-distance}
\end{equation}
Consequently, distance can be halved in
\begin{equation}
 O_p\left(1+\left(\frac{L_pr^{p-1}}\mu\right)^{2/(p+1)}\right)
 \label{eq:base-halving}
\end{equation}
queries. This is a distance guarantee, not an unsupported tangent-residual claim at the average.

An explicit sufficient iteration count is
\begin{equation}
 T_p(r,\mu)=\max\left\{1,
 \left\lceil\left[2b_p\left(\frac43\right)^{(p-1)/2}
                    \frac{L_pr^{p-1}}{\mu}\right]^{2/(p+1)}\right\rceil\right\}.
 \label{eq:halving-count}
\end{equation}
Indeed, \eqref{eq:lambda-sum} then implies $\Lambda_T\ge2/\mu$, and \eqref{eq:teg-distance} gives distance at most $r/2$.
\end{proof}

\begin{algorithm}[tbp]
\caption{\textsc{Halve}$(G,w,\mu,r)$: a finite distance-halving schedule}
\label{alg:halve}
\begin{algorithmic}[1]
\Require $w\in\mathcal X$, $r,\mu>0$; the promise $\norm{w-z^\star}\le r$ is used only for the guarantee.
\State $x\gets w$, $\Lambda\gets0$, $s_{\mathrm{avg}}\gets0$, $T\gets T_p(r,\mu)$ from \eqref{eq:halving-count}.
\For{$i=0,\ldots,T-1$}
 \State Query the jet of $G$ at $x$ and solve \eqref{eq:model-step} for $y$.
 \If{$y=x$} \State \Return $y$ \Comment{An exact VI solution.} \EndIf
 \State Query $G(y)$; set $\lambda\gets(\beta\norm{y-x}^{p-1})^{-1}$.
 \State $s_{\mathrm{avg}}\gets s_{\mathrm{avg}}+\lambda y$; $\Lambda\gets\Lambda+\lambda$.
 \State $x\gets\Pi_{\mathcal X}(x-\lambda G(y))$.
\EndFor
\State \Return $s_{\mathrm{avg}}/\Lambda$.
\end{algorithmic}
\end{algorithm}

Repeat this halving routine with promised radii $r,r/2,r/4,\ldots$ until the radius is at most $(\delta/(A_pL_p))^{1/p}$. Lemma~\ref{lem:correction} then returns a certificate of norm at most $\delta$. The geometric sum of the nonconstant halving costs is bounded by a constant times the first one. Thus for a \emph{valid supplied radius} $r$, the cost is
\begin{equation}
 O_p\left(1+\left(\frac{L_pr^{p-1}}\mu\right)^{2/(p+1)}
       +\log_+\frac{A_pL_pr^p}{\delta}\right),
 \label{eq:base-known-radius}
\end{equation}
where $\log_+(s)=\max\{0,\log s\}$.
To make the summation explicit, the number of halving stages is at most
\[
 1+\frac{1}{p\log 2}\log_+\frac{A_pL_pr^p}{\delta},
\]
and their nonconstant costs are bounded by the geometric sum
\[
 \left(\frac{L_pr^{p-1}}\mu\right)^{2/(p+1)}
 \sum_{j\ge0}2^{-2j(p-1)/(p+1)}.
\]
These are bounds on the \emph{scheduled} query count even when the
promised radius is false. Only the distance guarantee uses the promise.
The final certificate must therefore be checked before accepting output.

\begin{algorithm}[tbp]
\caption{\textsc{BaseTrial}$(G,w,\mu,R,\delta)$: a certified trial with a possibly false radius promise}
\label{alg:base-trial}
\begin{algorithmic}[1]
\State $z\gets w$; $r\gets R$.
\While{$A_pL_pr^p>\delta$}
 \State $z\gets\textsc{Halve}(G,z,\mu,r)$; $r\gets r/2$.
\EndWhile
\State At $z$, perform the model correction of Lemma~\ref{lem:correction}, obtaining $(y,n,G(y))$.
\If{$\norm{G(y)+n}\le\delta$} \State \Return $(y,n,G(y))$. \EndIf
\State \Return $\Fail$.
\end{algorithmic}
\end{algorithm}
Even when $R$ is false, Algorithm~\ref{alg:base-trial} has a finite schedule. Its normals remain genuine because model solving does not depend on the distance promise. A false promise may cause a failed norm test, but never an accepted false certificate.

\section{Certifying interfaces and unknown initial distances}
\label{sec:interfaces}
The warm-start distance will be a quantity in the analysis, not input information. This distinction is essential. We now formalize a distance-adaptive interface and keep explicit logarithms so that nested precision costs cannot hide a power of $\eps^{-1}$.

For $0<R\le D$ and $\eps>0$, define
\begin{align}
 Q_R&=\frac{L_pR^p}{\eps},&
 \LA(R,\eps)&=1+\log\left(2+\frac DR+Q_R\right),\label{eq:LA}\\
 \kappa_D&=\frac{L_pD^{p-1}}\mu,&
 \LS(\mu,\delta)&=1+\log\left(2+\kappa_D+
          \frac{L_pD^p}{\delta}+\frac{\mu D}{\delta}\right).
 \label{eq:LS}
\end{align}
All quantities inside logarithms are dimensionless.

\begin{definition}[Two algorithm interfaces]
\label{def:interfaces}
A residual solver $\mathcal A_{\alpha,b}$, on input $(F,x_0,R,\eps)$, has a \emph{distance promise}: some VI solution $z^\star$ satisfies $\norm{x_0-z^\star}\le R$. Under this promise it returns a feasible normal certificate of norm at most $\eps$ using
\begin{equation}
 C_p(1+Q_R^\alpha)\LA(R,\eps)^b
 \label{eq:A-interface}
\end{equation}
queries. It is budget capped: if its prescribed budget is exhausted without a verified certificate, it reports failure. It never returns an unverified successful output, even under a false promise. Every active cap is charged before a new original-oracle query, including queries in all descendants. Exhaustion immediately aborts that capped call; it is not a check performed after an uncontrolled subcall has finished.

A strong solver $\mathcal S_{\gamma,b}$, on input $(G,w,\mu,\delta)$, assumes $G$ is $\mu$-strongly monotone, but is \emph{not} given $r=\norm{w-z^\star}$. It returns a feasible normal certificate of norm at most $\delta$ using
\begin{equation}
 C_p\left[1+\left(\frac{L_pr^{p-1}}\mu\right)^\gamma\right]
       \LS(\mu,\delta)^b
 \label{eq:S-interface}
\end{equation}
queries. The convention at $r=0$ is the additive constant in this expression. Each interface has a fixed, known leading constant attached to its implementation. These constants are chosen recursively from the estimates proved below, independently of the instance. A returned certificate also carries the already evaluated operator value at its point.
\end{definition}

\begin{proposition}[A base strong solver with no supplied distance]
\label{prop:base-strong}
The construction in Section~\ref{sec:base} yields
$\mathcal S_{2/(p+1),2}$.
\end{proposition}
\begin{proof}
Let $\rho=(\mu/L_p)^{1/(p-1)}$. Starting each trial from the original point $w$, try radii
\[
 R_j=\min\{D,2^j\rho\},\qquad j=0,1,\ldots,
\]
stopping at the first verified certificate; do not repeat $D$. For each trial run Algorithm~\ref{alg:base-trial}, using its finite schedule and final certificate test. The first radius at least $r$ succeeds. If $\rho\ge D$ there is just one trial, and $\kappa_D\le1$. Otherwise the largest required trial radius is at most $2\max\{r,\rho\}$. Hence its condition quantity is at most a $p$-dependent constant times $1+L_pr^{p-1}/\mu$.

There are $O_p(\LS)$ trials. In each, the logarithmic term in \eqref{eq:base-known-radius} is $O_p(\LS)$, since $R_j\le D$. Summing even the crude bound ``number of trials times largest trial cost'' gives \eqref{eq:S-interface} with $\gamma=2/(p+1)$ and $b=2$. No trial assumes that the unknown true distance has been observed.
\end{proof}

\begin{lemma}[A residual solver can be restarted into a strong solver]
\label{lem:restart}
Existence of $\mathcal A_{\alpha,b}$ implies existence of
$\mathcal S_{\alpha,b+2}$ for every fixed $\alpha>0$ occurring below.
\end{lemma}
\begin{proof}
First suppose that a radius $R$ is valid. For a $\mu$-strongly monotone VI, any feasible normal certificate satisfies
\begin{equation}
 \mu\norm{x-z^\star}\le\norm{G(x)+n},\qquad n\in N_{\mathcal X}(x).
 \label{eq:strong-certificate-distance}
\end{equation}
Indeed, strong monotonicity of $G+N_{\mathcal X}$ and Cauchy--Schwarz give this inequality after pairing with $x-z^\star$.

For $R_j=R/2^j$, run $\mathcal A_{\alpha,b}$ at accuracy $\mu R_j/2$, and use its output as the next center. This halves the true distance whenever the current promise is valid. At the first stage where $\mu R_j/2\le\delta$, instead run the final call at accuracy $\delta$ and return its certificate. The number of stages is $O(1+\log_+(\mu R/\delta))$. Every call has $Q_{R_j}\le2L_pR_j^{p-1}/\mu$, so the nonconstant costs sum geometrically.

To remove the initial distance promise, use the same guesses
$R=\min\{D,2^\ell\rho\}$ as in Proposition~\ref{prop:base-strong}, with $\rho=(\mu/L_p)^{1/(p-1)}$. On a false promise a subcall can hit its cap and fail; that trial is then abandoned. Every trial has a predetermined finite sum of caps, and a final successful certificate is checked. On the first valid guess no cap is reached prematurely, so the trial succeeds.

For these guessed radii, $D/R\le\max\{1,D/\rho\}$. Before the final accuracy stage, the smallest radius is bounded below, up to a factor two, by $\delta/\mu$; if no halving is needed, the initial guessed radius gives the bound instead. Consequently all the logarithms $\LA(R_j,\eps_j)$ are $O_p(\LS(\mu,\delta))$. There are $O_p(\LS)$ stages and $O_p(\LS)$ guesses. The largest required guess is at most $2\max\{r,\rho\}$, or $D$ if $\rho>D$. The resulting bound is precisely \eqref{eq:S-interface} with exponent $\alpha$ and logarithmic power $b+2$.

The caps are important: using a promised-distance method without a cap on an invalid guessed radius would not justify its trial cost.
\end{proof}

\begin{algorithm}[tbp]
\caption{\textsc{AdaptiveStrong}$(\mathcal T,G,w,\mu,\delta)$: remove a radius promise}
\label{alg:adaptive-strong}
\begin{algorithmic}[1]
\Require A finite certified trial $\mathcal T(G,w,\mu,R,\delta)$, valid whenever $R\ge\norm{w-z^\star}$.
\State $\rho\gets(\mu/L_p)^{1/(p-1)}$; $R\gets\min\{D,\rho\}$.
\Loop
 \State Run $\mathcal T(G,w,\mu,R,\delta)$ from the original point $w$.
 \If{the trial returns a verified certificate of norm at most $\delta$}
  \State \Return that certificate, including its evaluated operator value.
 \EndIf
 \If{$R=D$} \State \Return $\Fail$ \Comment{Impossible on valid strong instances.} \EndIf
 \State $R\gets\min\{D,2R\}$.
\EndLoop
\end{algorithmic}
\end{algorithm}

\begin{algorithm}[tbp]
\caption{\textsc{RestartTrial}$(\mathcal A,G,w,\mu,R,\delta)$}
\label{alg:restart-trial}
\begin{algorithmic}[1]
\Require A budget-capped residual solver $\mathcal A$ with a distance promise.
\State $z\gets w$; $r\gets R$.
\Loop
 \State $\sigma\gets\max\{\delta,\mu r/2\}$.
 \State Run the capped call $\mathcal A(G,z,r,\sigma)$.
 \If{the call fails} \State \Return $\Fail$. \EndIf
 \State Let $(y,n,G(y))$ be its verified output.
 \If{$\sigma=\delta$} \State \Return $(y,n,G(y))$. \EndIf
 \State $z\gets y$; $r\gets r/2$.
\EndLoop
\end{algorithmic}
\end{algorithm}
Proposition~\ref{prop:base-strong} is Algorithm~\ref{alg:adaptive-strong} with $\mathcal T=\textsc{BaseTrial}$. Lemma~\ref{lem:restart} is the same wrapper with $\mathcal T=\textsc{RestartTrial}$. Each new trial restarts from $w$, rather than reusing an unchecked point from a failed trial. Neither algorithm restricts the original domain to an artificial ball; all normals belong to $N_{\mathcal X}$.

\section{Second-difference energy along a Halpern orbit}
\label{sec:energy}
Let $P$ be a firmly nonexpansive map with a fixed point $z^\star$. Firm nonexpansiveness means
\begin{equation}
 \norm{P(x)-P(y)}^2+
 \norm{(x-P(x))-(y-P(y))}^2\le\norm{x-y}^2.
 \label{eq:firm}
\end{equation}
Consider the exact Halpern orbit
\begin{equation}
 y_t=P(x_t),\qquad x_{t+1}=\frac{x_0+(t+1)y_t}{t+2},\qquad t\ge0,
 \label{eq:exact-halpern}
\end{equation}
where $\norm{x_0-z^\star}\le R$. Its points $x_t,y_t$ stay within distance $R$ of $z^\star$. Put $H_t=\sum_{i=1}^t1/i$ and $H_0=0$.

\begin{lemma}[Elementary first-difference and residual estimates]
\label{lem:firstdiff}
For $t\ge1$, write $b_t=x_t-x_{t-1}$ and $a_t=y_t-y_{t-1}$. Then
\begin{equation}
 \norm{a_t}\le\norm{b_t}\le\frac{2RH_t}{t+1},
 \qquad
 \norm{x_t-y_t}\le\frac{2R(1+H_{t+1})}{t+1}.
 \label{eq:firstdiff}
\end{equation}
For $t\ge2$, the exact identity
\begin{equation}
 b_t=\frac{t}{t+1}a_{t-1}+\frac{y_{t-2}-x_0}{t(t+1)}
 \label{eq:b-identity}
\end{equation}
holds. In particular, for $c_t=b_t-a_{t-1}$,
\begin{equation}
 \norm{c_t}\le\frac{2R(1+H_{t-1})}{t(t+1)}.
 \label{eq:c-bound}
\end{equation}
\end{lemma}
\begin{proof}
Subtract two consecutive Halpern updates to obtain \eqref{eq:b-identity}. We have $\norm{b_1}\le R$ and $\norm{y_t-x_0}\le2R$. Nonexpansiveness therefore gives
\[
 (t+1)\norm{b_t}\le t\norm{b_{t-1}}+\frac{2R}t.
\]
Induction proves the bound on $b_t$, and nonexpansiveness bounds $a_t$. Subtract $a_{t-1}$ from \eqref{eq:b-identity} to get \eqref{eq:c-bound}. Finally the update also gives
\[
 (t+1)(x_t-y_t)=x_0-x_t-(t+2)b_{t+1}.
\]
The first-difference bound and $\norm{x_0-x_t}\le2R$ prove the residual estimate. Its formula also holds for $t=0$ with the evident direct bound.
\end{proof}

\begin{lemma}[Blockwise squared second-difference estimate]
\label{lem:second-energy}
For every integer $k\ge2$ and $k\le\ell\le2k-1$,
\begin{equation}
 \boxed{\sum_{t=k}^{\ell}\norm{y_t-2y_{t-1}+y_{t-2}}^2
 \le64\,\frac{R^2(1+H_{2k})^2}{k^2}.}
 \label{eq:second-energy}
\end{equation}
\end{lemma}
\begin{proof}
Apply \eqref{eq:firm} to $(x_t,x_{t-1})$ and put $e_t=b_t-a_t$. Then
\[
 \norm{a_t}^2+\norm{e_t}^2\le\norm{b_t}^2
 =\norm{a_{t-1}+c_t}^2.
\]
Sum from $k$ to $\ell$. The $a_t$ terms telescope, giving
\begin{equation}
 \sum_{t=k}^{\ell}\norm{e_t}^2
 \le\norm{a_{k-1}}^2+
 \sum_{t=k}^{\ell}\bigl(2\norm{a_{t-1}}\norm{c_t}+\norm{c_t}^2\bigr).
 \label{eq:firm-telescope}
\end{equation}
Let $h=1+H_{2k}$. By Lemma~\ref{lem:firstdiff}, over this block,
\[
 \norm{a_{t-1}}\le\frac{2Rh}{t},\qquad
 \norm{c_t}\le\frac{2Rh}{t^2},\qquad
 \norm{a_{k-1}}\le\frac{2Rh}{k}.
\]
There are at most $k$ summands. Thus the right side of \eqref{eq:firm-telescope} is at most
\[
 \frac{4R^2h^2}{k^2}+\frac{8R^2h^2}{k^2}
       +\frac{4R^2h^2}{k^3}
 \le\frac{14R^2h^2}{k^2}.
\]
Since $a_t-a_{t-1}=c_t-e_t$,
\[
 \sum_{t=k}^{\ell}\norm{a_t-a_{t-1}}^2
 \le2\sum_{t=k}^{\ell}\norm{e_t}^2+2\sum_{t=k}^{\ell}\norm{c_t}^2
 \le\frac{32R^2h^2}{k^2}.
\]
This implies the displayed, more conservative constant $64$.
\end{proof}

\begin{remark}[An aggregate gain, not a pointwise acceleration claim]
The bound gives root-mean-square prediction error $O(R\log k/k^{3/2})$ on a block of $k$ iterates. It does \emph{not} assert that every second difference is this small. Firm nonexpansiveness, not mere nonexpansiveness, supplies the dissipative term in \eqref{eq:firm-telescope}. The map $P$ is fixed throughout the orbit.
\end{remark}

\section{Extrapolated inexact Halpern iteration}
\label{sec:inexact}
For $A=F+N_{\mathcal X}$ and a fixed $\eta>0$, define
\[
 P_\eta=(I+\eta A)^{-1}.
\]
For every center $x\in\R^d$, the continuous regularized operator
$z\mapsto F(z)+(z-x)/\eta$ has a VI solution on $\mathcal X$ by the
same compact fixed-point argument as in Section~\ref{sec:main}.
It is $1/\eta$-strongly monotone; comparing two solutions in the
VI inequalities proves uniqueness. Hence $P_\eta$ is well defined
as a map $\R^d\to\mathcal X$.
If $u=P_\eta(x)$ and $v=P_\eta(y)$, their defining inclusions and
monotonicity of $F+N_{\mathcal X}$ give
\[
 \ip{(x-u)-(y-v)}{u-v}\ge0.
\]
Expanding $\norm{x-y}^2$ as the square of
$(u-v)+((x-u)-(y-v))$ proves \eqref{eq:firm}.
Also $P_\eta(z^\star)=z^\star$ precisely when $z^\star$ is a VI
solution. The resolvent is used to analyze the algorithm; it is not
available as a free oracle operation.

Fix a supplied distance bound $R\le D$, target $\eps$, and integer $K\ge2$. Set
\begin{equation}
 e=\frac{R}{(K+2)^4},\qquad
 \eta=\frac{8R(1+H_{K+1})}{\eps(K+1)},\qquad
 \delta=\frac e\eta.
 \label{eq:outer-params}
\end{equation}
\textbf{The stepsize $\eta$ is kept constant during this complete outer run.}

Starting from $x_0$, for $t=0,1,\ldots,K$ do the following. Use the feasible warm start
\begin{equation}
 w_0=x_0,\qquad w_1=\widehat y_0,\qquad
 w_t=\Proj(2\widehat y_{t-1}-\widehat y_{t-2})\quad(t\ge2).
 \label{eq:predictor}
\end{equation}
Define the strongly monotone operator
\begin{equation}
 G_t(z)=F(z)+\frac{z-x_t}{\eta},\qquad \mu=\frac1\eta.
 \label{eq:inner-G}
\end{equation}
Invoke an already constructed solver $\mathcal S_{\gamma,b}$ at $(G_t,w_t,\mu,\delta)$. It returns $(\widehat y_t,n_t)$ with
\begin{equation}
 n_t\in N_{\mathcal X}(\widehat y_t),\qquad
 v_t=G_t(\widehat y_t)+n_t,\qquad \norm{v_t}\le\delta.
 \label{eq:inner-certificate}
\end{equation}
For $t<K$, make the ordinary Halpern update
\begin{equation}
 x_{t+1}=\frac{x_0+(t+1)\widehat y_t}{t+2}.
 \label{eq:actual-halpern}
\end{equation}
Return $(\widehat y_K,n_K)$, not the bare Halpern point $x_K$.
All inner original-oracle calls count. One jet of $G_t$ is obtained from one jet of $F$ plus known affine data. Since $p\ge2$, this affine addition does not increase $L_p$.

\begin{algorithm}[tbp]
\caption{\textsc{ExtrapolatedHalpern}$(\mathcal S_{\gamma,b},F,x_0,R,\eps)$}
\label{alg:eh}
\begin{algorithmic}[1]
\State $\alpha\gets\gamma/(1+\gamma)$; $K\gets\max\{2,\lceil(L_pR^p/\eps)^\alpha\rceil\}$.
\State Choose $e,\eta,\delta$ by \eqref{eq:outer-params}; keep $\eta$ fixed.
\State Install the interruptible original-query cap $\mathsf B$ in \eqref{eq:explicit-cap}.
\For{$t=0,\ldots,K$}
 \State Form $w_t$ using \eqref{eq:predictor} and $G_t(z)=F(z)+(z-x_t)/\eta$.
 \State Obtain $(\widehat y_t,n_t,G_t(\widehat y_t))$ from $\mathcal S_{\gamma,b}(G_t,w_t,1/\eta,\delta)$.
 \If{the active cap is exhausted} \State \Return $\Fail$. \EndIf
 \If{$t<K$} \State $x_{t+1}\gets[x_0+(t+1)\widehat y_t]/(t+2)$. \EndIf
\EndFor
\State Compute $F(\widehat y_K)=G_K(\widehat y_K)-(\widehat y_K-x_K)/\eta$ from saved data.
\If{$\norm{F(\widehat y_K)+n_K}\le\eps$}
 \State \Return $(\widehat y_K,n_K,F(\widehat y_K))$.
\EndIf
\State \Return $\Fail$.
\end{algorithmic}
\end{algorithm}
The cap is implemented as a shared query counter through every descendant call. The indicated exhaustion test describes immediate interruption, not a delayed test after the subroutine returns. No additional query is required to transform the saved $G_K$ value into the final $F$ value.

\subsection{Stability against an exact shadow orbit}
Let $(\bar x_t,\bar y_t)$ be the exact orbit \eqref{eq:exact-halpern} for the same $P_\eta$ and initial point. A VI solution within distance $R$ of $x_0$ is a fixed point, so Section~\ref{sec:energy} applies.
From strong monotonicity and \eqref{eq:inner-certificate},
\begin{equation}
 \norm{\widehat y_t-P_\eta(x_t)}\le\eta\delta=e.
 \label{eq:prox-error}
\end{equation}
Induction using nonexpansiveness yields
\begin{equation}
 \norm{x_t-\bar x_t}\le\frac{te}2,\qquad
 \norm{\widehat y_t-\bar y_t}\le\frac{(t+2)e}2.
 \label{eq:shadow}
\end{equation}
Indeed, if the first bound holds at $t$, the next difference is at most
$\frac{t+1}{t+2}(te/2+e)=(t+1)e/2$.

\begin{lemma}[True warm-start distances inherit the energy bound]
\label{lem:warm-energy}
Let $r_t=\norm{w_t-P_\eta(x_t)}$, an analysis quantity not supplied to the algorithm. Then $r_0,r_1\le3R$, and for $t\ge2$,
\begin{equation}
 r_t\le\norm{\bar y_t-2\bar y_{t-1}+\bar y_{t-2}}+(2t+1)e.
 \label{eq:warm-comparison}
\end{equation}
In particular, for $k\ge2$ and $k\le\ell\le\min\{2k-1,K\}$,
\begin{equation}
 \sum_{t=k}^{\ell}r_t^2\le C\frac{R^2(1+H_{2k})^2}{k^2}
 \label{eq:actual-warm-energy}
\end{equation}
with the universal choice $C=256$.
\end{lemma}
\begin{proof}
Because $P_\eta(x_t)\in\mathcal X$, projection in \eqref{eq:predictor} cannot increase the distance to it. Add and subtract exact shadow points and use \eqref{eq:shadow} and nonexpansiveness. The accumulated error is at most
\[
 2\frac{(t+1)e}{2}+\frac{te}{2}+\frac{te}{2}=(2t+1)e,
\]
which proves \eqref{eq:warm-comparison}. For $t=0$, nonexpansiveness around a fixed point gives $r_0\le2R$. For $t=1$, use $r_1\le e+\norm{x_1-x_0}\le R+3e/2\le3R$.

Square \eqref{eq:warm-comparison}, sum, and apply Lemma~\ref{lem:second-energy}. The additional term is at most $50k^3e^2$. Since $k\le K$ and $e=R/(K+2)^4$, it is at most $50R^2/k^2$, so it is absorbed in \eqref{eq:actual-warm-energy}. More explicitly, Lemma~\ref{lem:second-energy} and $(a+b)^2\le2a^2+2b^2$ give the constant $128+50\le256$.
\end{proof}

\begin{lemma}[The actual returned point has a tangent certificate]
\label{lem:final-cert}
Suppose the distance promise is valid and all $K+1$ inner calls
complete with the certificates in \eqref{eq:inner-certificate}.
Then their final point and normal satisfy
$\norm{F(\widehat y_K)+n_K}\le\eps$.
The noninterruption of a valid budget-capped run is established
in Theorem~\ref{thm:boost}.
\end{lemma}
\begin{proof}
The returned normal belongs to $N_{\mathcal X}(\widehat y_K)$, and
\[
 F(\widehat y_K)+n_K=v_K-\frac{\widehat y_K-x_K}{\eta}.
\]
Combining \eqref{eq:shadow}, \eqref{eq:inner-certificate}, and Lemma~\ref{lem:firstdiff} gives
\begin{equation}
 \norm{F(\widehat y_K)+n_K}
 \le\frac1\eta\left[\frac{2R(1+H_{K+1})}{K+1}+(K+2)e\right].
 \label{eq:output-bound}
\end{equation}
With \eqref{eq:outer-params} the first term is $\eps/4$ and the second is less than $\eps/4$. This proves the claim with slack. The last inner solve has already been counted; no exact proximal recovery is omitted.
\end{proof}

\section{One-step exponent improvement and logarithmic accounting}
\label{sec:boost}
The upper procedure must be finite even if its supplied radius is invalid, because it will itself be used in radius-guessing trials. The following cap makes that requirement operational. For $a>0$ and $K\ge2$, define the known quantity
\begin{equation}
 U_a(K)=2\cdot3^a+16^a\sum_{j=1}^{\lfloor\log_2 K\rfloor}
       (1+H_{2^{j+1}})^a\,2^{j(1-3a/2)}.
 \label{eq:Ua}
\end{equation}
If the strong-solver interface has leading constant $C_{\mathcal S}$, install the cap
\begin{equation}
 \mathsf B=
 \left\lceil C_{\mathcal S}\LS(1/\eta,e/\eta)^b
 \left[K+1+(\eta L_p)^\gamma R^a U_a(K)\right]\right\rceil+1,
 \qquad a=(p-1)\gamma.
 \label{eq:explicit-cap}
\end{equation}
All quantities in this expression are supplied parameters or finite sums of known scalars. In particular, no unknown warm-start distance appears. The proof below shows that a valid run finishes within this cap. A false radius promise can only cause a checked failure. Formula~\eqref{eq:explicit-cap} also specifies a concrete choice of cap without access to an unknown optimal complexity.

\begin{theorem}[Extrapolated Halpern exponent transformation]
\label{thm:boost}
Suppose $\mathcal S_{\gamma,b}$ exists, with $0<\gamma\le1$, and
\begin{equation}
 a=(p-1)\gamma\le2,\qquad \frac32a\ge1.
 \label{eq:boost-range}
\end{equation}
Then the construction above, with
\begin{equation}
 \alpha=\frac{\gamma}{1+\gamma},\qquad
 K=\max\{2,\lceil Q_R^\alpha\rceil\},
 \label{eq:K-choice}
\end{equation}
produces $\mathcal A_{\alpha,b+4}$.
\end{theorem}
\begin{proof}
Consider first the run without its outer cap. Every $G_t$ is
$1/\eta$-strongly monotone on the same domain, so the strong-solver
interface guarantees completion of every inner call. Feasibility
follows from projection and convex combinations, and its final
certificate is accurate by Lemma~\ref{lem:final-cert}.
We now bound the total unknown-operator queries of this run, including
those used to discover adequate warm-start radii. The resulting bound
will show that the installed cap never interrupts a valid run.

For $0<a\le2$, the finite-dimensional power-mean inequality and \eqref{eq:actual-warm-energy} imply, on a dyadic block,
\begin{align}
 \sum_{t=k}^{\ell}r_t^a
 &\le k^{1-a/2}\left(\sum_{t=k}^{\ell}r_t^2\right)^{a/2}\notag\\
 &\le C_pR^a(1+H_{2k})^a k^{1-3a/2}.
 \label{eq:block-power}
\end{align}
With the constant $256$ in \eqref{eq:actual-warm-energy}, \eqref{eq:block-power} has coefficient $16^a$. Including the two initial distances and summing the actual dyadic blocks therefore proves
\[
 \sum_{t=0}^{K}r_t^a\le R^a U_a(K).
\]
Together with the strong-solver interface, this proves that the cap \eqref{eq:explicit-cap} does not interrupt a valid run. By \eqref{eq:boost-range}, the power of $k$ is nonpositive. Sum over $k=2,4,8,\ldots$, include the two initial distances, and allow an extra logarithm at equality. This gives
\begin{equation}
 \sum_{t=0}^K r_t^a\le C_p R^a[1+\log(K+2)]^{a+1}.
 \label{eq:sum-power}
\end{equation}
The true-distance-sensitive strong-solver interface now yields
\begin{equation}
 N\le C_p\LS(1/\eta,e/\eta)^b
 \left[K+1+(\eta L_p)^\gamma\sum_{t=0}^K r_t^{(p-1)\gamma}\right].
 \label{eq:cost-pre}
\end{equation}
This expression includes the $t=0$ and $t=1$ startup solves. There is no claim that each inner problem has constant cost.

Write $h=1+H_{K+1}$. The dimensionless terms in the strong-solver logarithm are exactly
\begin{align*}
 \eta L_pD^{p-1}&=\frac{8hQ_R}{K+1}\left(\frac DR\right)^{p-1},\\
 \frac{L_pD^p}{e/\eta}&=\frac{8hQ_R(K+2)^4}{K+1}\left(\frac DR\right)^p,\\
 \frac{(1/\eta)D}{e/\eta}&=\frac DR(K+2)^4.
\end{align*}
Since $K$ in \eqref{eq:K-choice} is polynomially bounded in $1+Q_R$, these identities imply
\begin{equation}
 \LS(1/\eta,e/\eta)\le C_p\LA(R,\eps),
 \qquad 1+\log(K+2)\le C_p\LA(R,\eps).
 \label{eq:log-control}
\end{equation}
Also $\eta L_pR^{p-1}=8hQ_R/(K+1)$. Substituting \eqref{eq:sum-power} and \eqref{eq:log-control} into \eqref{eq:cost-pre}, and using $a\le2$, $\gamma\le1$, gives
\begin{equation}
 N\le C_p\LA(R,\eps)^{b+4}
       \left[K+1+\left(\frac{Q_R}{K+1}\right)^\gamma\right].
 \label{eq:boost-cost}
\end{equation}
The choice \eqref{eq:K-choice} balances the two powers: $\gamma(1-\alpha)=\alpha$. Hence \eqref{eq:A-interface} holds with $\alpha=\gamma/(1+\gamma)$ and logarithmic power $b+4$.

The explicit cap \eqref{eq:explicit-cap} satisfies the same bound because it uses precisely the block majorant whose logarithms were estimated above. Under a valid promise it does not interfere, and under a false promise it bounds every descendant's contribution to the enclosing call. The final certificate is accepted only after its norm is verified. This supplies the budget-capped interface required by Lemma~\ref{lem:restart}.
\end{proof}

\section{Finite acceleration hierarchy: proof of the upper bound}
\label{sec:bootstrap}
\begin{proof}[Proof of the upper bound in Theorem~\ref{thm:main}]
Proposition~\ref{prop:base-strong} constructs a base strong solver with
\[
 \gamma_0=\frac2{p+1},\qquad b_0^{\rm S}=2.
\]
For $j=1,\ldots,p-1$, apply Theorem~\ref{thm:boost} to the already constructed strong solver, and, except at the last level, apply Lemma~\ref{lem:restart} to the resulting residual solver. Algebra gives
\begin{equation}
 \gamma_j=\frac{\gamma_{j-1}}{1+\gamma_{j-1}}
          =\frac2{p+1+2j},\qquad
 b_j^{\rm A}=6j,\qquad b_j^{\rm S}=6j+2.
 \label{eq:gamma-recursion}
\end{equation}
The strong exponent at the input to the $j$th boost satisfies
\[
 (p-1)\gamma_{j-1}=\frac{2(p-1)}{p+2j-1}\le2,
 \qquad
 \frac32(p-1)\gamma_{j-1}=\frac{3(p-1)}{p+2j-1}\ge1
\]
for every $1\le j\le p-1$. Equality in the second inequality occurs at the final level. Thus every use of the boost is within its proved range.

After exactly $p-1$ levels,
\begin{equation}
 \gamma_{p-1}=\frac2{3p-1},\qquad b_{p-1}^{\rm A}=6(p-1).
 \label{eq:final-alpha}
\end{equation}
Take $R=D$ in the resulting residual solver. Then $\LA(D,\eps)=1+\log(3+Q)$, proving the final upper bound in \eqref{eq:matching}.

This is finite induction, not an invocation of the unknown optimal solver inside itself: each new residual solver calls only a lower-level strong solver, and the base solver was constructed directly from Taylor models. Nested regularizations add only known affine terms. Every jet of any nested operator is obtained by one query to the original $F$, and all such calls are included in the inductive budgets. There is no uncounted stronger oracle.
\end{proof}

For illustration, the residual exponents produced in successive levels are
\[
\begin{array}{c|c|c}
p&\text{base strong exponent}&\text{successive residual exponents}\\\hline
2&2/3&2/5\\
3&1/2&1/3,\ 1/4\\
4&2/5&2/7,\ 2/9,\ 2/11\\
5&1/3&1/4,\ 1/5,\ 1/6,\ 1/7.
\end{array}
\]
For $p=2$, only one acceleration level is needed: projected linear prediction and a base strong solver of exponent $2/3$ yield the cost
\[
 \Ot\left(K+\left(\frac{L_2R^2}{\eps K}\right)^{2/3}\right),
\]
which is minimized at $K\asymp(L_2R^2/\eps)^{2/5}$.

\begin{corollary}[Distance-sensitive upper bound]
Under a valid supplied distance bound $0<R\le D$, the final residual solver satisfies \eqref{eq:distance-upper} and returns an evaluated normal certificate. Under an invalid distance bound it remains a finite capped procedure and never reports an unverified success.
\end{corollary}
\begin{proof}
Use the interface $\mathcal A_{\gamma_{p-1},6(p-1)}$ before the substitution $R=D$ in the preceding proof.
\end{proof}

The hierarchy is algorithmically finite: construct $\mathcal S_0$ from Algorithms~\ref{alg:base-trial} and~\ref{alg:adaptive-strong}; define $\mathcal A_j$ by Algorithm~\ref{alg:eh} using $\mathcal S_{j-1}$; and, for $j<p-1$, define $\mathcal S_j$ by Algorithms~\ref{alg:restart-trial} and~\ref{alg:adaptive-strong}. Leading constants are attached to these proven interfaces in this order. The depth depends only on the fixed $p$, not on accuracy or dimension.

\begin{algorithm}[tbp]
\caption{\textsc{HigherOrderVI}: the final finite hierarchy}
\label{alg:main}
\begin{algorithmic}[1]
\Require Public data $(p,L_p,D,\mathcal X,x_0,\eps)$, with $p\ge2$.
\Ensure An evaluated certificate $(x,n,F(x))$ with
$n\in N_{\mathcal X}(x)$ and $\norm{F(x)+n}\le\eps$.
\State Instantiate $\mathcal S_0$ as \textsc{AdaptiveStrong} with
\textsc{BaseTrial}.
\For{$j=1,\ldots,p-1$}
 \State Define $\mathcal A_j$ as \textsc{ExtrapolatedHalpern} using
              the already defined $\mathcal S_{j-1}$.
 \If{$j<p-1$}
  \State Define $\mathcal S_j$ as \textsc{AdaptiveStrong} with
                  \textsc{RestartTrial} using $\mathcal A_j$.
 \EndIf
\EndFor
\State Run $\mathcal A_{p-1}(F,x_0,D,\eps)$ and return its certificate.
\end{algorithmic}
\end{algorithm}
In Algorithm~\ref{alg:main}, ``define'' means instantiate a callable
routine with the constants and interruptible query caps established
above, not solve an additional unknown problem. Every active ancestor
counter is charged before the next original-$F$ query. By the valid
promise $\norm{x_0-z^\star}\le D$, the final call cannot exhaust its
cap before obtaining its certificate. For a zero-dimensional affine
hull, use the single-point procedure in Appendix~\ref{app:affine}.

\section{A smooth lower-bound family}
\label{sec:lower}
The following construction uses the exact-flatness mechanism of smooth shift lower bounds~\cite{jangryu}, but normalizes both smoothness and domain diameter before the oracle argument. It will be used for arbitrary queries, not a prescribed tensor-step algorithm class.

\subsection{A gate with an exact information plateau}\label{sec:lower-gate}
Choose once and for all an even $C^\infty$ function $\rho:\R\to[0,1]$ satisfying
\[
\rho(u)=0\quad(|u|\le1/2),\qquad
\rho(u)=1\quad(|u|\ge1).
\]
For example, define $\eta(t)=e^{-1/t}$ for $t>0$ and $0$ otherwise, set
$\chi(t)=\eta(t)/(\eta(t)+\eta(1-t))$, and take
$\rho(u)=\chi((4u^2-1)/3)$. The denominator defining $\chi$ never vanishes.
Set
\begin{equation}
 B_p=\max\{1,\norm{\rho^{(p-1)}}_\infty\},\qquad
 h_\tau(s)=\tau\int_0^{s/\tau}\rho(u)\,du\quad(\tau>0).
\label{eq:gate}
\end{equation}

\begin{lemma}[Gate properties]
\label{lem:gate}
The function $h_\tau$ is odd and $C^\infty$. For every $s\in\R$,
\begin{equation}
0\le h_\tau'(s)\le1,\quad
|h_\tau(s)|\le|s|,\quad
|s-h_\tau(s)|\le\tau,\quad
h_\tau(s)\ge s-\tau.
\label{eq:gate-basic}
\end{equation}
It is identically zero on $[-\tau/2,\tau/2]$, and all its derivatives vanish there. Moreover,
\begin{equation}
\Lip(h_\tau^{(p-1)})\le B_p\tau^{1-p}.
\label{eq:gate-smooth}
\end{equation}
On the positive tail, $h_\tau(s)=s-b_\tau$ for $s\ge\tau$, where
\[
 b_\tau=\tau\int_0^1(1-\rho(u))\,du\in[\tau/2,\tau].
\]
\end{lemma}
\begin{proof}
Differentiating gives $h_\tau'(s)=\rho(s/\tau)$ and
$h_\tau^{(k)}(s)=\tau^{1-k}\rho^{(k-1)}(s/\tau)$ for $k\ge1$.
The derivative bound implies the Lipschitz and sign claims. For $s\ge0$,
\[
0\le s-h_\tau(s)=\tau\int_0^{s/\tau}(1-\rho(u))\,du\le\tau.
\]
Oddness handles negative $s$. The tail and flatness statements follow from the definition of $\rho$, including at the endpoints by smoothness. Finally,
$\norm{h_\tau^{(p)}}_\infty\le B_p\tau^{1-p}$ implies \eqref{eq:gate-smooth}.
\end{proof}

\subsection{A fixed-diameter monotone family}
In this section the public domain is $\mathcal X=B_2^d(R)$,
with the public initial point $x_0=0$. Neither depends on the hidden
frame below. For a query budget $N\ge1$, write $m=N+1$ and set
\begin{equation}
 q=1-\frac1{2m},\qquad c=\frac{R}{2\sqrt m},\qquad
 \tau=\frac{c}{8m}=\frac{R}{16m^{3/2}},\qquad
 \lambda=\frac{L_p\tau^{p-1}}{B_p}.
\label{eq:parameters}
\end{equation}
Let $V=(v_1,\ldots,v_m)$ be an orthonormal frame in $\R^d$. Define
\begin{equation}
 T_V(x)=c v_1+q\sum_{i=1}^{m-1}h_\tau(\ip{v_i}{x})v_{i+1},
 \qquad
 F_V(x)=\lambda\bigl(x-T_V(x)\bigr).
\label{eq:operator}
\end{equation}
All scalar parameters, the gate, and even this template may be disclosed to the algorithm; only the frame is hidden.

\begin{lemma}[Contraction, smoothness, and a uniformly interior zero]
\label{lem:operator}
The following properties hold independently of $d$ and the frame.
\begin{enumerate}
\item $T_V$ is globally $q$-contractive and maps $\Ball$ into itself.
\item $F_V$ is globally $\lambda(1-q)$-strongly monotone and satisfies \eqref{eq:class} with the prescribed $L_p$.
\item Its unique zero $x^\star$ belongs to $\Ball$ and satisfies
\[
\frac{3R}{16}\le\norm{x^\star}\le\frac{R}{\sqrt3}.
\]
\item For every $x\in\Ball$, including boundary points,
\begin{equation}
\rtan(x)=\norm{F_V(x)}.
\label{eq:no-cancellation}
\end{equation}
\end{enumerate}
\end{lemma}
\begin{proof}
Orthogonality and the scalar Lipschitz bound give
\[
\norm{T_V(x)-T_V(y)}^2
 =q^2\sum_{i=1}^{m-1}|h_\tau(\ip{v_i}{x})-h_\tau(\ip{v_i}{y})|^2
 \le q^2\norm{x-y}^2.
\]
Also, for $\norm{x}\le R$,
\[
\norm{T_V(x)}^2\le c^2+q^2R^2\le R^2,
\]
since $c^2/R^2=1/(4m)$ and
$1-q^2=1/m-1/(4m^2)\ge1/(4m)$.

By Cauchy--Schwarz,
\[
\ip{F_V(x)-F_V(y)}{x-y}
\ge\lambda(1-q)\norm{x-y}^2.
\]
For unit vectors $u_1,\ldots,u_p$,
\[
 D^pT_V(x)[u_1,\ldots,u_p]
 =q\sum_{i=1}^{m-1}h_\tau^{(p)}(\ip{v_i}{x})
       \prod_{j=1}^p\ip{v_i}{u_j}\,v_{i+1}.
\]
Its squared norm is at most
\[
q^2 B_p^2\tau^{2-2p}
\sum_i\prod_{j=1}^p|\ip{v_i}{u_j}|^2
\le q^2 B_p^2\tau^{2-2p}.
\]
The last inequality follows by bounding all but one factor by $1$ and using orthonormality for the remaining factor. Since $p\ge2$, the identity term in $F_V$ has zero $p$th derivative. Thus
\[
\norm{D^pF_V(x)}_{\rm op}\le\lambda qB_p\tau^{1-p}\le L_p.
\]
Integration along a segment proves the required Lipschitz bound for $D^{p-1}F_V$. In particular, no dimension factor is hidden in $L_p$.

The contraction theorem supplies a unique fixed point of $T_V$. More explicitly,
\[
 x^\star=\sum_{i=1}^m s_i v_i,\qquad
 s_1=c,\quad s_{i+1}=q h_\tau(s_i).
\]
The gate bounds imply $0\le s_i\le cq^{i-1}$, and hence
\[
\norm{x^\star}^2\le\frac{c^2}{1-q^2}
 =\frac{R^2}{4(1-1/(4m))}\le\frac{R^2}{3}.
\]
Conversely, $h_\tau(s)\ge s-\tau$ and Bernoulli's inequality give
\[
 s_i\ge cq^{i-1}-q\tau\sum_{j=0}^{i-2}q^j
 \ge c/2-m\tau=3c/8.
\]
Thus $\norm{x^\star}\ge(3c/8)\sqrt m=3R/16$.

For an interior $x$, its normal cone is $\{0\}$. On the boundary,
$N_{\Ball}(x)=\{a x:a\ge0\}$ and ball invariance yields
\[
\ip{F_V(x)}{x}
=\lambda\bigl(R^2-\ip{T_V(x)}{x}\bigr)\ge0.
\]
Consequently $\norm{F_V(x)+a x}^2\ge\norm{F_V(x)}^2$ for every $a\ge0$, with equality at $a=0$. This proves \eqref{eq:no-cancellation}.
\end{proof}

\begin{remark}[Strong monotonicity does not change the problem class]
The hard instances happen to be strongly monotone, but their strong-monotonicity parameter is $\lambda/(2m)$ and tends to zero with $N$. We do not impose a fixed positive strong-monotonicity constant on the full class. The construction therefore remains a lower bound for the original merely monotone setting.
\end{remark}

\subsection{A robust residual lower bound on a hidden slab}
\begin{lemma}[Hidden final coordinate]
\label{lem:residual}
For every $x\in\Ball$ satisfying $|\ip{v_m}{x}|\le\tau/2$,
\begin{equation}
\rtan(x)\ge\frac{\lambda R}{8m}
 =\frac{L_pR^p}{8B_p16^{p-1}\,m^{(3p-1)/2}}.
\label{eq:hard-residual}
\end{equation}
\end{lemma}
\begin{proof}
Write $s_i=\ip{v_i}{x}$ and $e=x-T_V(x)$. The coordinates of $e$ satisfy
\[
 e_1=s_1-c,\qquad
 e_i=s_i-qh_\tau(s_{i-1})\le s_i-qs_{i-1}+q\tau\quad(i\ge2).
\]
Multiply by $w_i=q^{m-i}$ and sum. The coordinate terms telescope:
\begin{align*}
 \sum_{i=1}^m w_i e_i
 &\le s_m-cq^{m-1}+q\tau\sum_{i=2}^m q^{m-i}\\
 &\le\tau/2-c/2+m\tau
 =\tau/2-3c/8\le-c/4.
\end{align*}
Here $q^{m-1}\ge1/2$, $m\tau=c/8$, and $\tau/2\le c/16$.
The vector $w=\sum_iw_iv_i$ has norm at most $\sqrt m$. Therefore
\[
 \norm{x-T_V(x)}\ge\frac{|\ip{w}{e}|}{\norm{w}}
 \ge\frac{c}{4\sqrt m}=\frac{R}{8m}.
\]
Multiplication by $\lambda$, Lemma~\ref{lem:operator}, and the parameter values prove the claim. Components of $x$ orthogonal to the frame do not invalidate this argument.
\end{proof}

\section{Adaptive lower bounds}
\label{sec:lower-bounds}
\subsection{Unrestricted deterministic algorithms}\label{sec:lower-det}
\begin{theorem}[Unrestricted deterministic exact-jet lower bound]
\label{thm:det}
For every $p\ge2$, $L_p,R>0$, $N\ge1$, dimension $d\ge2N+2$, and deterministic algorithm using at most $N$ feasible queries to the exact oracle, there is a frame $V$ such that its arbitrary feasible output $\widehat x$ on $F_V$ obeys \eqref{eq:hard-residual}. In particular, with $D=2R$,
\[
\mathsf T_p^{\rm det}(\eps;L_p,D)
 =\Omega_p\!\left((L_pD^p/\eps)^{2/(3p-1)}\right)
\]
in the high-dimensional worst-case model.
\end{theorem}
\begin{proof}
Pad an algorithm using fewer queries with ignored feasible queries. Let its queries be $x_1,\ldots,x_N$. We construct the frame while answering queries, and then exhibit a single fixed operator consistent with the full transcript.

At query $t$, choose $v_t$ orthogonal to
\[
\Span\{v_1,\ldots,v_{t-1},x_1,\ldots,x_t\}.
\]
The dimension of this span is at most $2t-1<d$. Return all oracle derivatives through order $p-1$ of
\begin{equation}
 F_t(x)=\lambda\left(x-cv_1-q\sum_{i=1}^{t-1}
                   h_\tau(\ip{v_i}{x})v_{i+1}\right)
 \label{eq:truncated-operator}
\end{equation}
at $x_t$, with an empty sum for $t=1$. Since the algorithm is deterministic, its next query, and eventually its output $\widehat x$, are now determined.

After the last response, choose $v_m=v_{N+1}$ orthogonal to
\[
\Span\{v_1,\ldots,v_N,x_1,\ldots,x_N,\widehat x\}.
\]
This span has dimension at most $2N+1<d$. Define the final operator to be $F_V$ in \eqref{eq:operator}.

For each previous query $x_t$, all $v_i$ with $i\ge t$ are orthogonal to $x_t$. Every term missing from $F_t$ has the form
$h_\tau(\ip{v_i}{x})v_{i+1}$ with $i\ge t$; its value and all derivatives vanish at $x_t$. Hence
\[
 D^kF_V(x_t)=D^kF_t(x_t),\qquad 0\le k\le p-1.
\]
This is equality of the \emph{entire returned tensors}, not only of their products with previously queried directions. Thus the simulated transcript is exactly the transcript of a single fixed admissible operator. Finally, $\ip{v_m}{\widehat x}=0$, so Lemma~\ref{lem:residual} applies. Inverting its power law gives the query lower bound.
\end{proof}

\subsection{Randomized algorithms: a fixed instance distribution}\label{sec:lower-random}
The preceding resisting oracle is not by itself a randomized lower bound. We now draw the complete frame at the start and use the flat region of the gate to preserve exact transcript equality with high probability.
The symbol $\delta$ in this subsection is a failure-probability
parameter, separate from the inner-solver accuracy used in the upper
bound. All probabilities below are over a frame chosen before the run
and, when present, the algorithm's independent seed.

\begin{lemma}[A spherical tail inequality]
\label{lem:sphere}
If $k\ge1$, $v$ is uniform on the unit sphere of a
$k$-dimensional subspace, and $x$ is fixed with $\norm{x}\le R$, then,
for every threshold $a>0$,
\[
\Prob\{|\ip{v}{x}|>a\}\le2\exp\left(-\frac{k a^2}{2R^2}\right).
\]
\end{lemma}
\begin{proof}
Let $P$ be the orthogonal projection onto the subspace.
For completeness, take a standard Gaussian vector $g\in\R^k$ and
write $g=S u$, where $S=\norm{g}$ and $u$ is uniform on the unit sphere.
The direction and radius are independent. Gaussian integration and the
radial integral give
\[
 \E g_1^{2\ell}=(2\ell-1)!!,\qquad
 \E S^{2\ell}=k(k+2)\cdots(k+2\ell-2).
\]
Dividing these identities gives the even moments of $u_1$;
rotational invariance then yields
\[
\E\ip{v}{x}^{2\ell}
=\norm{P x}^{2\ell}\frac{(2\ell-1)!!}{k(k+2)\cdots(k+2\ell-2)}
\le R^{2\ell}\frac{(2\ell-1)!!}{k^\ell}.
\]
Odd moments vanish, so summing the exponential series yields
$\E e^{t\ip{v}{x}}\le e^{t^2R^2/(2k)}$. Chernoff's bound on both signs gives the assertion.
\end{proof}

\begin{theorem}[Bounded-query randomized lower bound]
\label{thm:random}
Let $m=N+1$, fix $0<\delta<1$, and suppose
\begin{equation}
 d\ge m+\left\lceil2048m^3\log\frac{2(m^2+1)}{\delta}\right\rceil.
\label{eq:random-dimension}
\end{equation}
Draw $V$ uniformly from the orthonormal $m$-frames in $\R^d$ and use \eqref{eq:operator}. For every randomized algorithm making at most $N$ feasible oracle queries and returning a feasible $\widehat x$,
\begin{equation}
 \Prob_{V,\,\mathrm{seed}}\left\{
 \rtan(\widehat x)\ge
 \frac{L_pR^p}{8B_p16^{p-1}m^{(3p-1)/2}}
 \right\}\ge1-\delta.
\label{eq:random-hard}
\end{equation}
Consequently some \emph{fixed} frame has this failure probability over the algorithm's seed alone. Taking $\delta=1/4$ proves
\[
\mathsf T_p^{\rm rand}(\eps;L_p,D;\,\mathrm{success}\ge2/3)
 =\Omega_p\!\left((L_pD^p/\eps)^{2/(3p-1)}\right)
\]
with dimension allowed to grow. An expected-residual lower bound follows as well.
\end{theorem}
\begin{proof}
Use the total-feasible rules from Section~\ref{sec:main} and pad any stopped run by ignored queries, retaining its saved output. These rules are defined and bounded on every simulated transcript, whether or not that transcript is eventually coupled to an actual instance. Condition first on an arbitrary algorithm seed. Define a \emph{simulated} run: at query $t$, return the derivatives of the truncated operator $F_t$ from the deterministic proof. The simulated query $x_t^{\rm sim}$ depends only on $v_1,\ldots,v_{t-1}$; the answer may additionally depend on $v_t$. In particular the simulated output depends only on $v_1,\ldots,v_N$.

Define the events
\[
 \begin{aligned}
 E_{t,i}&=\{\,|\ip{v_i}{x_t^{\rm sim}}|\le\tau/2\,\}
       &&(1\le t\le N,\ t\le i\le m-1),\\
 E_{\rm out}&=\{\,|\ip{v_m}{\widehat x^{\rm sim}}|\le\tau/2\,\},\qquad
 E=E_{\rm out}\cap\bigcap_{t=1}^N\bigcap_{i=t}^{m-1}E_{t,i}.
 \end{aligned}
\]
Conditioned on $v_1,\ldots,v_{t-1}$, each remaining $v_i$ is marginally
uniform on the unit sphere of their $(d-t+1)$-dimensional orthogonal
complement. Lemma~\ref{lem:sphere} therefore gives, for $i=t,\ldots,m-1$,
\[
\Prob\{|\ip{v_i}{x_t^{\rm sim}}|>\tau/2\}
\le2\exp\left(-\frac{(d-m)\tau^2}{8R^2}\right).
\]
Likewise, conditioning on $v_1,\ldots,v_N$ gives the same bound for
$|\ip{v_m}{\widehat x^{\rm sim}}|>\tau/2$. A union bound over fewer than $m^2+1$ events shows that their simultaneous complement has probability at least
\[
1-2(m^2+1)\exp\left(-\frac{d-m}{2048m^3}\right)\ge1-\delta.
\]
No conditioning on previous success events is used, so the conditional spherical distribution is not incorrectly preserved after selection.

On $E$, every omitted gate at every simulated query is exactly
zero together with all its derivatives. We verify transcript equality
inductively, for the fixed seed. Before the first answer, both runs use
the same query rule and have the same empty history. If their histories
agree through query $t-1$, they choose the same point at query $t$.
For every omitted term indexed by $i\ge t$, $E_{t,i}$ places the scalar
argument in the exactly flat interval. Hence the full jet of $F_V$ at
that point equals the jet of $F_t$ in \eqref{eq:truncated-operator}.
The histories therefore agree through query $t$. After $N$ queries,
their saved outputs agree as well, including when the algorithm stopped
early. Event $E_{\rm out}$ places the common output in the slab of
Lemma~\ref{lem:residual}. Thus \eqref{eq:random-hard} holds for each fixed
seed. Integrating over the independent seed proves the randomized
statement.

Finally average \eqref{eq:random-hard} over $V$. There is a frame for which the conditional probability over seeds is at least $1-\delta$. When the threshold in \eqref{eq:random-hard} exceeds $\eps$ and $\delta=1/4$, this contradicts success probability $2/3$ on every instance. The expected-residual statement follows by multiplying the threshold by $1-\delta$.
\end{proof}

\begin{remark}[What the dimension assumption means]
The deterministic construction needs only $O(N)$ ambient dimension. The random-rotation proof above uses $O(N^3\log(N/\delta))$. It does not establish the same randomized lower bound in $O(N)$ dimension. Feasible-query boundedness is essential in the concentration proof: arbitrary unbounded probes outside the ball are not covered by Theorem~\ref{thm:random}.
\end{remark}

\subsection{Uniform normalization and completion of the theorem}
\begin{lemma}[At least one query is necessary]
\label{lem:zero-query}
For all positive $L_p,D,\eps$, no zero-query randomized algorithm can satisfy \eqref{eq:target} with success probability at least $2/3$ on every instance.
\end{lemma}
\begin{proof}
On $\mathcal X=[-D/2,D/2]$, let $F_+(x)=M$ and $F_-(x)=-M$, where $M>\eps$. Both are monotone and satisfy \eqref{eq:class}. Their $\eps$-solution sets are respectively the singleton left and right endpoints: at all other feasible points the tangent residual is $M$. A zero-query method has the same output distribution on both instances. Its two success probabilities sum to at most one, so they cannot both be at least $2/3$.
\end{proof}

\begin{proof}[Completion of the proof of Theorem~\ref{thm:main}]
Set $R=D/2$ in the lower family and define
\[
 c_p^{\mathrm{res}}=\frac{1}{2^p\,8B_p16^{p-1}}.
\]
Theorems~\ref{thm:det} and~\ref{thm:random}, with failure parameter $\delta=1/4$ in the latter, give an obstruction
\[
 \rtan(\widehat x)\ge
 c_p^{\mathrm{res}}\frac{L_pD^p}{(N+1)^{(3p-1)/2}}.
\]
For every $N\ge1$ with this threshold strictly larger than $\eps$, a method with $N$ queries fails the required success guarantee on a fixed instance. Therefore a uniform successful budget $T\ge1$ must obey
\[
 T\ge(c_p^{\mathrm{res}}Q)^{\alpha_p}-1.
\]
Together with Lemma~\ref{lem:zero-query}, this yields
\[
 T\ge\max\{1,(c_p^{\mathrm{res}}Q)^{\alpha_p}-1\}
 \ge\tfrac12(c_p^{\mathrm{res}})^{\alpha_p}Q^{\alpha_p}.
\]
This proves the lower bound uniformly, in particular for $Q\ge2$. A deterministic method is a randomized method that ignores its seed, so $\mathsf T_p^{\mathrm{rand}}\le\mathsf T_p^{\mathrm{det}}$. Section~\ref{sec:bootstrap} supplies the deterministic upper bound. All parts of \eqref{eq:matching} follow.
\end{proof}

\section{Discussion and limitations}
\label{sec:discussion}
\paragraph{What the theorem matches.}
Theorem~\ref{thm:main} determines the exponent of the tangent-residual
oracle complexity in the stated high-dimensional exact-jet model.  The
upper and lower bounds use the same smoothness parameter $L_p$, the same
diameter scaling $D^p$, the same feasible-query convention, and the same
strong solution criterion.  No scalar potential, fixed positive
strong-monotonicity parameter, span restriction, or tensor-update rule is
assumed for the unknown operator or for the lower-bound algorithm.

\paragraph{Scope of the Lean-based review.}
The system's Lean-based review covers the internal mathematical
developments and reductions presented in this manuscript. Results
imported from the literature are used through explicit interfaces
under their stated hypotheses; they are not treated as having been
re-proved or strengthened by the review.

\paragraph{Where geometry enters.}
The upper and lower proofs use geometry in different ways.  Firm
nonexpansiveness controls squared changes in the motion of the exact
proximal orbit, and projected extrapolation turns that energy control into
a sequence of useful warm starts.  In the lower bound, invariance of the
public Euclidean ball prevents boundary normals from canceling the hidden
coordinate obstruction.  Both ingredients are specific to producing an
actual same-point tangent certificate.

\paragraph{Information complexity versus computation.}
The upper theorem counts calls to the unknown operator, including every
nested evaluation and every final certificate check.  It does not bound
the arithmetic cost of solving a known regularized Taylor-model VI,
computing exact projections, or storing full derivative tensors.  Thus the
result should be interpreted as an information-complexity theorem rather
than a polynomial-time implementation guarantee.

\paragraph{Limitations.}
The logarithmic exponent $6(p-1)$ is conservative and is not claimed to
be optimal.  The lower bounds allow the dimension to grow with the query
budget; the randomized proof also assumes that all oracle queries are
feasible and hence bounded by the public domain.  The theorem does not
establish fixed-dimensional optimality, robustness to noisy or inexact
jets, or the same randomized lower bound for unrestricted outside-domain
queries.  Although the upper method applies to convex--concave saddle
operators, the lower family need not arise from a scalar saddle function,
so unrestricted optimality for that narrower subclass remains separate.

\paragraph{Matching convex--concave specialization.}
The scalar-potential specialization, together with the lower bound of Chen et
al.~\cite{minimax}, yields a matching exponent within the $p$th-order tensor
algorithm class of their Definition~5.1.  Recall that
$\TccTensor{p}(\delta)$ denotes complexity only within this class, not
unrestricted scalar full-jet minimax complexity.  The system's Lean-based
review covers the affine-hull reduction, the one-query simulation of an
operator $(p-1)$-jet by a scalar $p$-jet, the exact lift of the relative
normal certificate, and the admissibility map used to compare oracle/update
conventions.  The reduction gives
\begin{equation}
  \TccTensor{p}(\varepsilon_{\mathrm{tan}})
  \leq\widetilde O_p\!\left(
    \left(
      \frac{L_pD_Z^p}{\varepsilon_{\mathrm{tan}}}
    \right)^{2/(3p-1)}
  \right).
  \label{eq:discussion-cc-tan-upper}
\end{equation}
Theorem~5.2 of Chen et al.~\cite{minimax} gives, for the same class,
\begin{equation}
  \TccTensor{p}(\varepsilon_{\mathrm{tan}})
  \geq\Omega_p\!\left(
    \left(
      \frac{L_pD_Z^p}{\varepsilon_{\mathrm{tan}}}
    \right)^{2/(3p-1)}
  \right).
  \label{eq:discussion-cc-tan-lower}
\end{equation}
Thus
\begin{equation}
  \boxed{
  \TccTensor{p}(\varepsilon_{\mathrm{tan}})
  =\widetilde\Theta_p\!\left(
    \left(
      \frac{L_pD_Z^p}{\varepsilon_{\mathrm{tan}}}
    \right)^{2/(3p-1)}
  \right).}
  \label{eq:discussion-cc-tan-matching}
\end{equation}
Compared with the previously available upper exponent $4/(3p+1)$, the
transferred upper exponent is $2/(3p-1)$ and matches the known restricted
lower exponent.

For the duality-gap criterion studied in Chen et al.~\cite{minimax-colt},
$\operatorname{Gap}_\phi(z)\le D_Z\rtan(z)$ gives
\begin{equation}
  \TccTensor{p}(\varepsilon_{\mathrm{gap}})
  \leq\widetilde O_p\!\left(
    \left(
      \frac{L_pD_Z^{p+1}}{\varepsilon_{\mathrm{gap}}}
    \right)^{2/(3p-1)}
  \right).
  \label{eq:discussion-cc-gap-upper}
\end{equation}
The separate hard-family inequality in Lemma~5.1 of Chen et
al.~\cite{minimax}, followed by their rescaling in Eq.~(16), yields
\begin{equation}
  \TccTensor{p}(\varepsilon_{\mathrm{gap}})
  \geq\Omega_p\!\left(
    \left(
      \frac{L_pD_Z^{p+1}}{\varepsilon_{\mathrm{gap}}}
    \right)^{2/(3p-1)}
  \right).
  \label{eq:discussion-cc-gap-lower}
\end{equation}
Consequently,
\begin{equation}
  \boxed{
  \TccTensor{p}(\varepsilon_{\mathrm{gap}})
  =\widetilde\Theta_p\!\left(
    \left(
      \frac{L_pD_Z^{p+1}}{\varepsilon_{\mathrm{gap}}}
    \right)^{2/(3p-1)}
  \right).}
  \label{eq:discussion-cc-gap-matching}
\end{equation}

In the second-order setting, $p=2$ and $L_2=\rho$ is the
Hessian-Lipschitz modulus.  On the overlap of the assumptions used by the
conference theorem, its displayed leading term is
\[
  \underbrace{\widetilde O\!\left(
    D_X^{6/7}D_Y^{6/7}
    \left(\frac{L_2}{\varepsilon_{\mathrm{gap}}}\right)^{4/7}
  \right)}_{\text{previous CRN-oracle upper bound}},
\]
where lower-order regularity parameters enter suppressed logarithmic factors.
The later hard family gives
\[
  \underbrace{\Omega\!\left(
    \left(\frac{L_2D_Z^3}{\varepsilon_{\mathrm{gap}}}\right)^{2/5}
  \right)}_{\text{Definition~5.1 lower bound}}.
\]
The scalar specialization supplies the corresponding upper bound with
exponent $2/5$, so
\begin{equation}
  \boxed{
  \Omega\!\left(
    \left(\frac{L_2D_Z^3}{\varepsilon_{\mathrm{gap}}}\right)^{2/5}
  \right)
  \ \leq\ 
  \TccTensor{2}(\varepsilon_{\mathrm{gap}})
  \ \leq\ 
  \widetilde O\!\left(
    \left(\frac{L_2D_Z^3}{\varepsilon_{\mathrm{gap}}}\right)^{2/5}
  \right).}
  \label{eq:discussion-cc-sandwich}
\end{equation}
Equivalently,
\[
  \TccTensor{2}(\varepsilon_{\mathrm{gap}})
  =\widetilde\Theta\!\left(
    \left(
      \frac{L_2(D_X^2+D_Y^2)^{3/2}}
           {\varepsilon_{\mathrm{gap}}}
    \right)^{2/5}
  \right).
\]
Thus the previously available upper accuracy exponent $4/7$ is improved to
$2/5$, matching the lower exponent within the stated tensor algorithm class.
This is not an unrestricted lower-bound claim, and our reduction does not
enlarge the quantifiers in the cited theorem.

\paragraph{Open directions.}
Natural next questions are to reduce or remove the logarithmic overhead,
obtain dimension-explicit lower bounds, understand inexact and stochastic
higher-order oracles, and determine whether the same exponent is optimal
for convex--concave operators under an unrestricted full-jet algorithm
class.  Another algorithmic question is whether the Taylor-model
subproblems used here admit implementations with comparable arithmetic
complexity in structured problem classes.

\section{Conclusion}
We proved matching, up to logarithmic factors, deterministic and
randomized higher-order oracle bounds for the tangent residual of smooth
monotone variational inequalities.  The upper bound follows from a finite
hierarchy that repeatedly converts aggregate second-difference energy of
a Halpern orbit into certified warm-start savings.  The lower bound uses
exactly flat gates to hide complete jets from arbitrary adaptive queries.
Together, these arguments identify the exponent $2/(3p-1)$ in the stated
high-dimensional information model and isolate the remaining gaps that
belong to different criteria, algorithm classes, or computational models.
The research pipeline was carried out almost entirely by
\ResearchAgentSystem{}, which also conducted a Lean-based review of the
manuscript.

\appendix

\section{Affine-hull reduction and comparison of solution criteria}
\label{app:criteria}
\subsection{Reduction to the known affine hull}
\label{app:affine}
Suppose $\operatorname{aff}\mathcal X=x_c+U\R^r$, where $U^\top U=I_r$ is known. Define
\[
 \mathcal K=\{z:x_c+Uz\in\mathcal X\},\qquad
 \widetilde F(z)=U^\top F(x_c+Uz).
\]
Then $\mathcal K$ has the same diameter, $\widetilde F$ is monotone, and its derivative Lipschitz constant is at most $L_p$. One original oracle query supplies every projected tensor by applying $U$ to input directions and $U^\top$ to the output. If $\widetilde n\in N_{\mathcal K}(z)$ and $x=x_c+Uz$, set
\[
 n=U\widetilde n-(I-UU^\top)F(x).
\]
For every $y\in\mathcal X$, the displacement $y-x$ is in the range of $U$. Hence $\ip{n}{y-x}\le0$, so $n\in N_{\mathcal X}(x)$. Moreover,
\[
 F(x)+n=U(\widetilde F(z)+\widetilde n),\qquad
 \norm{F(x)+n}=\norm{\widetilde F(z)+\widetilde n}.
\]
The queried full value $F(x)$ is available at the final certified point. Thus the reduced algorithm lifts without loss in accuracy or oracle order. On a zero-dimensional domain, the unique point already has zero tangent residual. A point-only output needs no query; producing an evaluated normal certificate needs at most one query. This additive constant does not affect the stated bounds or the worst-case lower bound over all domains.

\subsection{Tangent residual, proximal residual, and gaps}
The resolvent facts used above are elementary instances of monotone operator calculus; see~\cite{ryuboyd}. For $A=F+N_{\mathcal X}$ and $y=P_\eta(z)$,
\[
 (z-y)/\eta\in A(y),\qquad
 \rtan(y)\le\norm{z-P_\eta(z)}/\eta.
\]
The point on the left is $y$, not $z$. Since both points are feasible, the right-hand side is at most $D/\eta$. Consequently allowing arbitrary $\eta$ would make a bare proximal-residual target uninformative without computing the proximal point. Algorithm~\ref{alg:eh} instead returns an actual approximate proximal point and its verified normal certificate; Lemma~\ref{lem:final-cert} includes its final inner solve.

\begin{proposition}[Computed approximate-proximal certificates]
\label{prop:approx-recovery}
Fix $z\in\mathcal X$, $\eta>0$, and $\delta\ge0$.
If $\widehat y\in\mathcal X$ and $n\in N_{\mathcal X}(\widehat y)$
satisfy
\[
 \norm{F(\widehat y)+n+(\widehat y-z)/\eta}\le\delta,
\]
then
\[
 \norm{\widehat y-P_\eta(z)}\le\eta\delta,\qquad
 \rtan(\widehat y)\le\frac{\norm{z-P_\eta(z)}}\eta+2\delta.
\]
In particular, one application of the model correction from
Lemma~\ref{lem:correction} to
$G(w)=F(w)+(w-z)/\eta$, centered at $z$, uses at most two original
queries and returns an evaluated tangent certificate of norm at most
\[
 \frac{R_z}{\eta}+2A_pL_pR_z^p,
 \qquad R_z=\norm{z-P_\eta(z)}.
\]
The analysis quantity $R_z$ need not be known to perform this correction.
\end{proposition}
\begin{proof}
Set $y^\star=P_\eta(z)$ and
$v=F(\widehat y)+n+(\widehat y-z)/\eta$.
Strong monotonicity of $F+N_{\mathcal X}+\eta^{-1}(I-z)$ gives
\[
 \frac1\eta\norm{\widehat y-y^\star}^2
 \le\ip{v}{\widehat y-y^\star}
 \le\delta\norm{\widehat y-y^\star}.
\]
This proves the distance bound, including the zero-distance case.
Consequently,
\[
 \begin{aligned}
 \norm{F(\widehat y)+n}
 &\le\delta+\frac{\norm{z-\widehat y}}\eta\\
 &\le\frac{\norm{z-y^\star}}\eta+2\delta.
 \end{aligned}
\]
For the last assertion, the solution of the regularized VI for $G$
is $y^\star$, its distance from the model center is $R_z$, and affine
regularization does not increase $L_p$. Lemma~\ref{lem:correction}
therefore gives a certificate for $G$ of norm at most
$\delta=A_pL_pR_z^p$. Substitute this value in the preceding bound.
\end{proof}

For $x\in\mathcal X$, define the Stampacchia and Minty gaps
\[
 g_{\mathrm S}(x)=\sup_{z\in\mathcal X}\ip{F(x)}{x-z},\qquad
 g_{\mathrm M}(x)=\sup_{z\in\mathcal X}\ip{F(z)}{x-z}.
\]
Monotonicity and the outward-normal convention give
\begin{equation}
 0\le g_{\mathrm M}(x)\le g_{\mathrm S}(x)\le D\rtan(x).
 \label{eq:gaps}
\end{equation}
Indeed, $\ip{F(z)}{x-z}\le\ip{F(x)}{x-z}$ and, for every $n\in N_{\mathcal X}(x)$,
\[
 \ip{F(x)}{x-z}\le\ip{F(x)+n}{x-z}\le D\norm{F(x)+n}.
\]
Take the supremum in $z$ and then the infimum in $n$.

For a convex--concave function $\phi$ on $\mathcal U\times\mathcal V$, let
$F(u,v)=(\nabla_u\phi(u,v),-\nabla_v\phi(u,v))$. Convexity and concavity imply
\[
 \phi(u,v')-\phi(u',v)
 \le\ip{\nabla_u\phi(u,v)}{u-u'}
    +\ip{-\nabla_v\phi(u,v)}{v-v'}.
\]
Thus its saddle gap is at most $g_{\mathrm S}(u,v)$ and therefore at most $D\rtan(u,v)$. The upper theorem consequently gives a gap-accuracy $\delta_g$ bound by taking $\eps=\delta_g/D$. A lower bound for the stronger tangent-residual target does not automatically lower-bound the weaker gap target.  The convex--concave statement in Section~\ref{sec:discussion} instead uses the separate gap inequality satisfied by the explicit hard family of Chen et al.; without that additional property, no matching-gap conclusion follows.

\end{document}